\documentclass[letterpaper ]{amsart}
\usepackage[english]{babel}
\usepackage{amsmath}
\usepackage{graphicx}
\usepackage[colorinlistoftodos]{todonotes}
\usepackage{amssymb}
\usepackage{extarrows}
\usepackage{amscd} 		
\usepackage{dsfont}
\usepackage{fixltx2e}
\usepackage{hyperref} 
\usepackage{cite}
\usepackage{tikz-cd}
\usepackage{listings}

\usepackage[margin=1in,top=0.75in,bottom=1in]{geometry} 

\newtheorem{theo}{Theorem}[section]
\newtheorem*{theo*}{Theorem} 
\newtheorem{lemm}{Lemma}
\newtheorem{assu}{Assumption}
\newtheorem{conj}{Conjecture}
\newtheorem*{conj*}{Conjecture}
\newtheorem{coro}{Corollary}
\newtheorem{prop}[theo]{Proposition}
\newtheorem{rema}{Remark}
\theoremstyle{definition} \newtheorem{exam}{Example} \newtheorem*{defi}{Definition}

\newcommand{\CC}{\mathbb{C}}

\newcommand{\PP}{\mathbb{P}}

\newcommand{\ZZ}{\mathbb{Z}}

\newcommand{\cI}{\mathcal{I}}
\newcommand{\cO}{\mathcal{O}}
\newcommand{\cL}{\mathcal{L}}

\newcommand{\cT}{\mathcal{T}}

\newcommand{\cZ}{\mathcal{Z}}
\newcommand{\cJ}{\mathcal{J}}
\newcommand{\cF}{\mathcal{F}}
\newcommand{\IC}{\mathcal{IC}}

\newcommand{\sm}{{Sm}}

\newcommand{\rk}{\operatorname{rk}}

\newcommand{\ind}{\mathds{1}}

\title{Geometric Invariants of Recursive Group Orbit Stratification}
\author{Xiping Zhang}
\email{xzhmath@gmail.com}
\address{SCMS, Fudan University, Shanghai,  China}
\thanks{MSC Classification: 14C17, 14J17, 32S05, 55S35}
\date{05/07/2021}

\begin{document}
\begin{abstract}
The local Euler obstructions and the Euler characteristics of  linear sections with all hyperplanes on a stratified projective variety  are   key geometric  invariants in the study of singularity theory. Despite their importance, in general it is very hard to compute them.
In this paper we consider a special type of singularity: the 
recursive group orbits. They are the group orbits of
a sequence of $G_n$ representations $V_n$ satisfying certain assumptions. 
We introduce a new intrinsic invariant called the $c_{sm}$ invariant, and use it to give explicit formulas to the local Euler obstructions and the  sectional Euler characteristics of such orbits. In particular, the matrix rank loci are examples of recursive group orbits. Thus as applications, we explicitly compute these geometry invariants for ordinary, skew-symmetric and symmetric rank loci.  Our method is systematic and algebraic, thus works for algebraically closed field of characteristic $0$. Moreover, in the complex setting we also compute the stalk Euler characteristics of the Intersection Cohomology Sheaf complexes for all three types of rank loci. 
\end{abstract}
\maketitle

\section{Introduction}
The  geometric invariants on singular varieties have been an important subject for us to understand the singularities. In 1973 MacPherson introduced a local measurement for complex varieties (later published as \cite{MAC}), and named it the  local Euler obstruction.  Defined as the obstruction to extend the distance $1$-form  after lifting to the Nash transform; 
it is the key ingredient in MacPherson's  proof of the existence and uniqueness of Chern class on singular spaces. An equivalent definition was given by Brasselet and Schwartz in \cite{B-S81} using vector fields.
Later Gonz\'{a}lez-Sprinberg showed in \cite{Gonzalez} that there is an algebraic formula for the local Euler obstruction function, thus this definition extends to arbitrary algebraically closed field. 
For more about local Euler obstructions we refer to  \cite{B-NG} and \cite{B-NG2}.

In the same year M. Kashiwara published his paper introducing the famous index theorem for holonomic D-modules \cite{Kashiwara73}. In the paper he defined certain local topological invariants,  as the weighted sum of the Euler characteristics of certain link spaces. He named them local characteristics. Although the two definitions were defined by two different flavors, and were introduced in two different branches of mathematics, it was proved in \cite{BDK} that surprisingly the two definitions are equivalent. As the key ingredients in both singularity theory and Kashiwara's index theorem, the local Euler obstructions of stratified spaces have been intensively studied in many different fields. It is one of the most important invariants in singularity theory. 

However, it is very hard to compute the local Euler obstructions in general. Many authors have been working on formulas that make the computation easier. For example, Gonz\'{a}lez-Sprinberg's formula allows one to use intersection theory method;  the formula in \cite{Le-Tessier81} reduced the computation to the knowledge of local polar multiplicities. In \cite{BLS} the authors provided a recursive formula using the Euler characteristics of the link spaces of strata.
 Recently in \cite{Rod-Wang2} the authors proved the relation between the local Euler obstructions and the maximal likelihood degree, and proposed an algorithm to compute local Euler obstructions by computer. 
Despite the difficulty, in many cases we have known the Euler obstructions very well. For Schubert cells in Grassmannnians they were computed in \cite{BFL}\cite{Jones}; for skew-symmetric and ordinary rank loci they were computed in \cite{R-Prom} and \cite{NG-TG}\cite{Xiping2}.  Recently in \cite{M-R20} the authors provide an algorithm  for local Euler obstructions of Schubert varieties in all cominuscule spaces. Based on the examples computed  Mihalcea and  Singh conjectured \cite[Conjecture 10.2]{M-R20} the non-negativity of the local Euler obstructions. 

Another fundamental geometric invariant is the Euler characteristic. Over $\CC$ it's simply the topological Euler characteristic, over arbitrary field of characteristic $0$ it can be defined via the integration of MacPherson's Chern class. When $X$ is a projective variety, we can access more refined invariants by consider linear sections, i.e., the Euler characteristic of  $X$ intersects with a given hyperplane $H$. We call them sectional Euler characteristics. When $H$ is generic, such invariants can be obtained from MacPherson's Chern class (Cf. \cite{Aluffi1}).
However, for `special' linear sections, the geometry of the intersection is determined by both the singularity of $X$ and the position of $H$, and concrete formulas are unknown in general. Also, it's quite subtle to determine when a hypersurface is generic. 
These sectional Euler characteristics are very important invariants of $X$. For example, 
when $X$ is a hypersurface, the generic sectional Euler characteristics of the complement correspond to another fundamental invariant of $X$:  the Milnor numbers of $X$ (Cf. \cite{Huh12}\cite{DP03}).  Also, from computational algebraic geometry perspective, many algorithms involves the step of `choose a generic hyperplane section', where one needs to compute the invariants on a generic linear section. For example, in \cite{Rod-Wang2} the authors prove a formula of local Euler obstructions in terms of the Maximal likelihood degrees of generic linear sections.
As  the case of local Euler obstructions, the sectional Euler characteristics are important but  very hard to compute in general .

In this paper we consider a special type of singularities: the recursive group orbits over algebraically closed field of characteristic $0$. By recursive group orbits we mean a sequence of group actions $G_n\to GL(V_n)$ satisfying certain conditions (Assumption~\ref{assum; orbits}). For such group orbits, we define an intrinsic invariant called the $c_{sm}$ invariants, using the Chern classes of the projectivized orbits.  
They are 
discussed  in \S\ref{S; EulerGeneral}. The main result in \S\ref{S; EulerGeneral} is Theorem~\ref{theo; EulerObstruction}:
\begin{theo*} 
Assume that we have sequence of group actions as in Assumption~\ref{assum; orbits}. For any $1\leq k\leq n-1$ we denote $S_{n,k}=\PP(\cO_{n,k})\subset \PP(V_n)$ to be the projectivizations. Denote the $c_{sm}$ invariants of $\cO_{n,k}$  by $\sm_{n,i}$.
We have the following formula for the local Euler obstructions :
\[
Eu_{\bar{\cO}_{n,k}}(\cO_{n,r})=Eu_{\cO_{r,k}}(0)=\sum_{(\mu_1>\mu_2>\cdots >\mu_l\geq k)} \sm_{n\mu_1}\cdot \sm_{\mu_1\mu_2}\cdots\cdot \sm_{\mu_l k}
\]
Here the sum is over all   (partial and full) flags $(\mu_1>\mu_2>\cdots >\mu_l)$ such that $\mu_i-1=\mu_{i+1}$ for every $i$ and $\mu_l\geq k$.
\end{theo*}
The formula  requires only the knowledge of our $c_{sm}$ invariants, which in many cases can be directly computed.  Our result is purely algebraic and works for general field $k$. In particular, when the base field is $\CC$,  we prove that our $c_{sm}$ invariants agree with Kashiwara's local Euler characteristics, i.e., 
the Euler characteristic of the complex link spaces of the orbits to $0$. We prove an algebraic version of \cite[Theorem 0.2]{Sch02} using our $c_{sm}$ invariants, which specialize to a Brasselet-Lê-Seade type formula (Proposition~\ref{prop; generalizedBLS}) .  

In \S\ref{S; SecEulerGeneral} we discuss the sectional Euler characteristics of group orbits satisfy Assumption~\ref{assum; orbits2}. The main result is Theorem~\ref{theo; sectionalEuler}, in which we
prove a formula to the Euler characteristics of such orbits with any given hyperplane in terms of their Euler obstructions:
\begin{theo*}
With the Assumption~\ref{assum; orbits2}. Let $e_{n,k}$ and $e'_{n,k}$ be their local Euler obstructions of the pair of orbits and their dual orbits.
Let $L_r\in S'_r$ be the corresponding hyperplane, and
let $A=[\alpha_{i,j}]_{n-1\times n-1}$ be the inverse matrix of the matrix $E=[e_{i,j}]_{n-1\times n-1}$ of local Euler obstructions.  Then we have
\begin{align*}
\chi_{S_k\cap L_r}=&
\sum_{i=r}^{n-1}\alpha_{k,i}\cdot (-1)^{d_i+d'_{n-i}+N} e'_{n-i,r}+\chi(S_k)-\sm(S_k)  \/.
\end{align*}
\end{theo*}
We also give explicit description on when a hypersurface is generic with respect to sectional Euler characteristic.  

In particular, for reflective recursive group orbits (group orbits satisfy Assumption~\ref{assum; orbits} and Assumption~\ref{assu; reflective}) we prove that, the local Euler obstructions of such orbits are equivalent to their sectional Euler characteristics, i.e., they can be deduced from each other. Thus for such group orbits we have proved that the three sets of local invariants-the local 
Euler obstructions, the sectional Euler characteristics and the $c_{sm}$ invariants-are equivalent.

As application we apply the formulas to a special type of reflective recursive group orbits: the ordianry, skew-symmetric and symmetric matrix rank loci. In \S\ref{S; rankloci} we briefly review the basic properties of the matrix rank loci, and in \S\ref{S; EulerRankLoci} we explicitly compute the local Euler obstructions for all three types rank loci. The  main results are Theorem~\ref{theo; OrdEuler}, Theorem~\ref{theo; SkewEuler} and Theorem~\ref{theo; SymEuler}.  When the base field is $\CC$, the formula is known for ordinary rank loci in \cite{NG-TG} and for skew-symmetric rank loci by \cite{Prom}, proved by concrete topological computation. It is not known for symmetric case. 
Our method, however, is systematic and algebraic: we translate the computation of the topological Euler characteristic of the link spaces into 
standard Schubert calculus: the computations of tautological  Chern classes over Grassmannians. Then we take advantage of the nice enumerative properties on Grassmannians. 

With the knowledge of local Euler obstructions, in  \S\ref{S; SecEulerRankLoci} we apply Theorem~\ref{theo; sectionalEuler} and compute the sectional Euler characteristics for the rank loci. For all three cases, we give explicit formulas to the Euler characteristics of the intersections of the rank loci with any given hyperplanes. The formulas in Theprem~\ref{theo; secEulerOrd}, Theorem~\ref{theo; secEulerskew} and Theorem~\ref{theo; sectionalSym} are summations of simple binomials, which can be easily carried out by hand.

The famous Kashiwara index theorem shows that the local Euler obstructions connect as a bridge from the microlocal multiplicities to the stalk Euler characteristics for constructible sheaf complexes. For stratified singular complex varieties we are particularly interested in the Intersection cohomology sheaf complexes: they are bounded complexes of constructible sheaves that reflect the limiting behaviors from strata to their boundaries.  As an application in 
Theorem~\ref{theo; stalkRankloci} we  compute the stalk Euler characteristic of the Intersection cohomology sheaves for all three types of rank loci. The computations indicate that, these topological invariant of ordinary rank loci behave very similarly to the skew-symmetric rank loci, but very differently from symmetric rank loci. For further discussions of this phenomenon we refer to  \cite{Braden} and \cite{BFL}.

As pointed in \cite{NG-TG}, for ordinary rank loci  the local Euler obstructions are, surprisingly, Newton binomials fitting into the Pascal triangle. The skew-symmetric rank loci behaves the same with ordinary case. 
For symmetric case this is almost true, except that the Euler obstructions are zero on even size but odd rank rank loci. For other cases the local Euler obstructions are also Newton binomials. 
When the base field is $\CC$, via MacPherson's definitions this vanishing phenomenon says that for such cases we can always extend the distance $1$-form on the Nash transform. It should be interesting to study from topology why such cases are special. 
Moreover, as singularities in the rank stratification  are the same as singularities of certain Schubert varieties in the Grassmannians (for symmetric and skew-symmetric case the Grassmannians should be chosen as Grassmannians of isotropic subspaces  with a symplectic or symmetric bilinear form), our results in \S\ref{S; EulerRankLoci}  support  the Conjecture 10.2 in \cite{M-R20}.

\section*{Acknowledgment}
The author is grateful to Leonardo Mihalcea, Richard Rim\'anyi, Fei Si, Changjian Su and Dingxin Zhang
for many useful discussions and suggestions.  The author also would like to thank the referees for all the comments. The author is supported by China Postdoctoral Science Foundation (Grant No.2019M661329).

\section{Preliminary}
\label{S; preliminary}
Unless specifically mentioned, all varieties in this note are  closed projective, and the base field $k$ is algebraically closed of characteristic $0$. 
\subsection{Characteristic Classes}
Let $X\subset \PP(V)$ be a closed subvariety. 
A constructible function on $X$ is a finite sum $\sum_W m_W \ind_W$ over closed subvarieties $W$ of $X$, where the coefficients $m_W$ are integers and  $\ind_W$ is the indicator function on $W$.
Let $F(X)$ be the addition group of constructible functions on $X$. It is isomorphic to the group of cycles $\cZ(X)$.

Let $f\colon X\to Y$ be a proper morphism. One defines a homomorphism $Ff\colon F(X)\to F(Y)$ by setting, for all $ p\in Y$, $
Ff(\ind_W)(p) := \chi(f^{-1}(p)\cap W)$,  
and extending this definition by linearity. This makes $F$ into 
a functor from the category of $k$-varieties $\mathcal{VAR}$ to $\mathcal{AB}$, the category of abelian groups. The following theorem was first proved in \cite{MAC} by MacPherson, to answer a conjecture proposed by Deligne and Grothendieck, then generalized to characteristic $0$ algebraically field by G. Kennedy in \cite{Kennedy}. 
\begin{theo}[MacPherson, Kennedy]
\label{theo; MACc_*}
Let $X$ be a projective variety.
There is a unique natural transformation $c_*$ from the functor
$F$ to the Chow functor $A$ such that if $X$ is smooth, then 
$c_*(\ind_X)=c(T_X)\cap [X]$, where $T_X$ 
is the tangent bundle of $X$.
\end{theo}

The  proof of the theorem may be summarized as the following steps.
\begin{enumerate}
\item For any closed projective variety $X$, there is a local measurement of the singularity called local Euler obstruction. This is 
a constructible function on $X$ denoted by $Eu_X$, i.e., $Eu_X=\sum_{W} e_W \ind_W$ for some sub-varieties $W$ of $X$.   
\item For any $X$, the Euler obstruction functions along all closed subvarieties $\{Eu_W \}$ form a basis for $F(X)$. 
\item For any closed subvariety $W$, there is an assigned characteristic class $c_M(W)\in A_*(W)$ defined via Nash blowup, called Chern-Mather class. 
\item Let $i\colon W\to X$ be the closed embedding. Define $c_*(Eu_W)=i_*(c_M(W))$ to be the pushforward of the Chern-Mather class of $W$ in $A_*(X)$. This is the unique natural transformation that matches the desired normalization property. 
\end{enumerate}

The Euler obstruction was originally defined over $\CC$ via obstruction theory to extend the lift of a non-zero vector field to the Nash transform, we refer \cite{MAC} for details. In \cite{Gonzalez} the author  gave an algebraic formula  works for  arbitrary algebraically closed field. 
Here we  use the formula in \cite{Gonzalez} as our definition. Let $X\subset M$ be a closed subvariety of dimension $d$, we define the Nash transform of $X$ to be 
\[
\hat{X}:=\text{closure of }\{(x,T_x X )|x \text{ is a smooth point} \}\subset G_d(TM) 
\]
The projection map $p\colon \hat{X}\to X$ is birational and is isomorphic over the smooth locus $X_{sm}$. The universal subbundle $S$ of $G_d(TM)$ restricts to a rank $d$ vector bundle on $\hat{X}$, denoted by $\cT_X$. 
\begin{defi}
The local Euler obstruction of $X$ is defined as follows: for any $x\in X$ we define
\[
Eu_X(x):=\int_{p^{-1}(x)} c(\cT_X)\cap s(p^{-1}(x), \hat{X}) \/.
\]
The Chern-Mather class of $X$ is defined as
\[
c_M(X):= p_*(c(\cT_X)\cap [\hat{X}]) \/.
\]
The Chern-Schwartz-MacPherson class of $X$ is defined as $c_{sm}(X):=c_*(\ind_X)$. Thus we have
\[
c_{sm}(X)=\sum_{k} a_k c_M(W_k)
\]
providing that $\ind_X=\sum_k a_k Eu_{W_k}$.
\end{defi}

\begin{prop}
We have the following properties.
\begin{enumerate}
\item  $Eu_X(x)$ is a local invariant, thus only depends on an open neighborhood of $x$ in $X$.
\item  If $x\notin X$, then $Eu_X(x)=0$.
\item  If $x\in X$ is a smooth point, then $Eu_X(x)=1$. But notice that $Eu_X(x)=1$ does NOT imply that $x$ is smooth (Cf. \cite{Nodland} and \cite{M-R20}, and Theorem~\ref{theo; SymEuler}).
\item  The Euler obstruction has the product property, i.e., $Eu_{X\times Y}(x\times y)=Eu_{X}(x)\times Eu_Y(y)$.
\end{enumerate}
\end{prop}

\begin{defi}
Let $X$ be a projective or affine variety. By a stratification on $X$ we mean a disjoint union $X=\cup_\alpha S_\alpha$ such that
\begin{enumerate}
\item Each $S_\alpha$ is a  smooth quasi-projective   variety.
\item For any $p$ and $q$ in $S_\alpha$, we have $Eu_{X}(p)=Eu_X(q)$.
\item If $S_\alpha\cap \bar{S}_\beta\neq \emptyset$, then we have $S_\alpha\subset \bar{S}_\beta$. 
\end{enumerate}
\end{defi}
\begin{exam}
 Let $X$ be a proejctive variety, we define $X_0=X_{sm}$ to be the smooth locus of $X$, and $\bar{X}_{k+1}$ to be the singularity of $\bar{X}_k$. This process will stop in finite steps, since $\bar{X}_{k+1}$ is closed in $\bar{X}_{k}$ and thus $\dim \bar{X}_{k+1}<\dim \bar{X}_k$. The strata $X_k:=\bar{X_k}\smallsetminus \bar{X}_{k+1}$ satisfy the above assumptions. We call this the  Singularity Stratification of $X$.
\end{exam}

\begin{rema}
Assuming that $X$ is complete. If we consider the constant map $k\colon X\to \{p\}$, then the covariance property of $c_*$
shows that
\begin{align*}
\int_X c_{sm}(Y)
=&~ \int_{\{p\}} Afc_*(\ind_Y)=\int_{\{p\}} c_*Ff(\ind_Y)\\
=&~ \int_{\{p\}} \chi(Y)c_*(\ind_{\{p\}})=\chi(Y). \qedhere
\end{align*}
This observation gives a generalization of the classical
Poincar\'e-Hopf Theorem to possibly singular varieties.
\end{rema} 

\subsection{Sectional Euler Characteristic}
In \cite{Aluffi1} Aluffi shows that, 
for projective varieties the generic sectional Euler characteristic can be obtained by studying their Chern-Schwzrtz-MacPherson classes. 
Let $X\subset \PP(V)$ be a projective variety of dimension $n$.  For any $r\geq 0$ we define 
\[
X_r=X\cap H_1\cap \cdots \cap H_r
\]
to be the intersection of $X$ with $r$ generic hyperplanes. Let $\chi(X_r)=\int_{X_r} c_{sm}(X_r)$ be its Euler characteristic, we define $\chi_X(t)=\sum_i \chi(X_r)\cdot (-t)^r$ to be the corresponding polynomial. This is a polynomial of degree $\leq n$. 

Also, notice that $A_*(\PP(V))=\ZZ[H]/H^{n+1}$ is a polynomial ring, thus we can write $i_*c_{sm}(X)=\sum_{i\geq 0} \gamma_{i} [\PP^i]=\sum_{i\geq 0} \gamma_{N-i} H^i$. 
Let $\gamma_X(t)=\sum_i \gamma_i t^i $ and $c_{sm}^X(t)=\sum_i \gamma_{N-i} t^i$ be the assigned polynomials respectively. They are polynomials of degree $\leq  n$. 

We have the following beautiful theorem connecting $c_{sm}$ class and sectional Euler chracteristics:
\begin{theo}[Aluffi]
\label{theo; EulerChern}
We consider the transformation $\cI_{SM} \colon \ZZ[t]^{\leq n}\to  \ZZ[t]^{\leq n}$ defined by
$$
P(t)\mapsto \cI_{SM}(P)(t):= \frac{t\cdot p(-t-1)+p(0)}{t+1} \/.
$$
This is a $\ZZ$-linear involution, i.e., $\cI_{SM}^2(P)=P$. Moreover, we have
$$
\cI_{SM}(\gamma_X(t))=\chi_X(t); \quad \cI_{SM}(\chi_X(t))=\gamma_X(t) \/.
$$
\end{theo}

In particular, we have the following property. 
\begin{prop}
\label{prop; csm(-1)}
For a generic hyperplane $H$, we have
$$
c_{sm}^X(-1)=(-1)^{n } (\chi(X)-\chi(X\cap H)) \/.
$$
\end{prop}
\begin{proof}
By the theorem, we have
\begin{align*}
\gamma_X(t)=& \frac{t\cdot \chi_X(-1-t)+\chi(0)}{t+1} \\
=& \frac{t\cdot (\chi_X(0)-\chi_1(-1-t)+\chi_2(-1-t)^2-\cdots ) +\chi_X(0)}{1+t} \\
=& \frac{(t+1)\chi_0+t\cdot \chi_1\cdot (t+1)+t\cdot \chi_2\cdot (1+t)^2+\cdots }{1+t} \\
=& \chi_0 +t\cdot \chi_1+t(1+t)\cdot F(t)
\end{align*}
Thus 
$$
c_{sm}^X(-1)=(-1)^n \cdot \gamma_X(-1)=(-1)^n (\chi(X)- \chi(X\cap H)) \/.
$$
\end{proof} 

Thus if we know the $c_{sm}$ class of $X$, then we can obtain the generic sectional Euler characteristics $\chi(X_r)$.
In general, the characteristic classes of a variety may be hard to compute. However, as in the property, we only need the alternating sum $c_{sm}^X(-1)$, which in practice is much easier to compute. 

\begin{defi}
Let $X\subset \PP(V)$ be a (quasi)-projective variety, and let $\Sigma\subset V$ be its affine cone. We define the $c_{sm}$ invariant of $X$ (or $\Sigma$) as
\[
\sm_{X}=\sm_{\Sigma}:=(-1)^{\dim V-1} c_{sm}^X(-1)=\chi(X)-\chi(X\cap H) \/.
\]
\end{defi}

In fact, if we write 
$
c_{sm}(X)=\sum_{i=0}^{\dim X} X_i ;
$
where $X_i\in A_i(X)$ denotes the dimension $i$ component of $c_{sm}(X)$. 
One can see that $\sm(X)=\sum_{i=0}^{\dim X} (-1)^i \deg X_i$. Thus this definition is intrinsic, i.e., independent of  the embedding of $X$.

\subsection{Dual Varieties and Radon Transform}
Let $X\in \PP^n=\PP(V)$ be a projective variety, and let $\PP^{*n}=\PP(V^*)$ be the (projective) dual space of hyperplanes in $\PP^n$, i.e.,  any point  in  $\PP(V^*)$ is a hyperplane in $\PP(V)$.
The dual variety of $X$, denoted by $X^\vee$, is defined as the closure of the following 
$$
\{H\in \PP^{*n}| \text{there is some }x\in X_{sm} \text{ such that } T_x X_{sm}\subset H\} \/.
$$
Here $X_{sm}$ denotes the smooth locus of $X$. The variety $X$ and its dual are related as follows. Consider the product $\PP(V)\times \PP(V^*)$ and define $I_X$ to be the closure of 
\[
I_X^\circ :=\{(x, h)|x\in X_{sm}; H\in \PP(V^*) \text{ s.t. } T_x X_{sm}\subset H\}
\]
to the the incidence correspondence. Let $p$ and $q$ denote the first and second projection. Then one can check that $p(I_X)=X$, and $q(I_X)=X^\vee$. The correspondence induces the definition of Radon transform:
\begin{defi}
The topological Radon transform is the group homomorphism
\[
F(X)\to F(X^\vee)\colon  \quad 
\lambda\mapsto \lambda^\vee :=q_*p^* \lambda \/.
\]
\end{defi}

If we write $\lambda=\sum_i n_i \ind_{W_i}$ , then by the definition of proper pushforward and pull back we have
\[
\lambda^\vee(L)=\sum_i  n_i \chi (W_i\cap L)
\]
for any point $L\in \PP(V^*)$.

In particular, as proved in \cite{Ernstrom}, the Radon transform takes the Euler obstruction function in $X$ to the Euler obstruction function in $X^\vee$, up to an error term.
\begin{theo}[Ernstr\"{o}m]
\label{theo: Radon}
Let $X$ be of dimension $n$, and $X^\vee$ be of dimension $m$. Let $V$ be of dimension $N+1$. Then 
$$
(Eu_X)^\vee  =(-1)^{n+(N-1)-m} Eu_{X^\vee}+ e_X \ind_{\PP(V^*)}
$$
\end{theo}
Here $e_X:=(Eu_X)^\vee(H)$ for generic hyperplanes $H$.  

This theorem shows that the Euler obstruction of the dual variety and the (signed) dual of the Euler obstruction differ by the Euler characteristic of generic hyperplane section.
\begin{rema}
Recall that there is a unique expression $\ind_X=\sum_i e_i Eu_{W_i}$ for closed subvarieties $W_i$. Then the generic value $e_x$ can be expressed as 
$$
e_X =\sum_i e_i \chi(W_i\cap H)
$$  
for a generic hyperplane $H$. Here by generic we mean that, by Bertini's theorem there is an open dense subset $U$ in $\PP(V^*)$ such that $\chi(W_i\cap H)$ is constant for any $H\in U$. 
\end{rema}

\begin{coro}
\label{coro; generic}
Let $X$ and $X^\vee$ be a pair of dual varieties. Assume that  $X=\cup_{i=1}^m S_i$ is a stratification. 
Let $L$ be a hyperplane such that $L$ lives outside $X^\vee$. Then $L$ is Euler characteristic generic with respect to $X$, i.e., 
$$
\chi(X\cap L)=\chi(X\cap H)
$$
for generic hyperplanes $H\in U$. 
In particular, if $L\notin \bar{S}_i^\vee$, then $L$ is Euler characteristic generic with respect to $S_i$:  $\chi(S_i\cap L)=\chi(S_i\cap H)$.

Moreover, when $X^\vee$ admits a stratification $\cup T_i$, we have 
$
\chi(X\cap L_1)=\chi(X\cap L_2)
$
for any $L_i$ and $L'_i$ living in the same stratum $T_i $.
\end{coro}
\begin{proof}
Recall that $l\notin X^\vee$ implies that $Eu_{X^\vee} (l)=0$. 
Let $e_j$ be the local Euler obstruction $Eu_{X}(S_j)$, then   the theorem says:
\[
(Eu_X)^\vee(L)=\sum_j e_j \chi(S_j\cap L)=0+\sum_i e_j \chi(S_i\cap H)
\]
Since there are only  finitely many strata, by induction on the dimension of strata it suffice to prove the argument for smooth $X$,  otherwise  its singularity locus induces a smaller stratum. 
When $X$ is smooth, we have $Eu_{X}=\ind_X$, and the theorem shows $\chi(X\cap L)=\chi(X\cap H)$. 

When $X^\vee=\cup T_i$ is a stratification, and $\{L_i, L_i'\}\subset T_i$ are two hyperplanes of $\PP(V)$.
The theorem shows that $(Eu_X)^\vee(L_i)=(Eu_X)^\vee(L'_i)$. Thus we have
\[
\sum_j e_j \chi(S_j\cap L_i) = \sum_j e_j \chi(S_j\cap L'_i) \/.
\]
Then the same induction  on the  dimensions of $S_i$ completes the proof. 
\end{proof}

\subsection{Projective Duality}

As pointed out in \cite[Prop 3.20]{Aluffi1}, the two-way projections from conormal space induce  an involution of 
Chern-Mather classes of dual varieties:
\begin{theo}[Aluffi]
Let $X\subset \PP(V)$ be a proper projective variety of dimension $d$, and let $X^\vee$ be its dual variety of dimension $d^*$. We write $c_M^X(H)\in A_*(\PP(V))$ and $c_M^{X^\vee}(H)\in A_*(\PP(V^*))$ as polynomials in $H$. Then 
\[
\cJ_n((-1)^{d}c_M^X(H))=(-1)^{d^*}c_M^{X^\vee}(H); \quad \cJ_n((-1)^{d^*}c_M^{X^\vee}(H))=(-1)^{d}c_M^X(H)
\]
Here $\dim V=n+1$, and $\cJ_n\colon \ZZ[H]\to \ZZ[H]$  is defined as 
\[
\cJ_n(p(H))=p(-1-H)-p(-1)\left( (1+H)^{n+1}-H^{n+1}  \right) \/.
\]
\end{theo}

As corollary of the duality theorem, we have the following result.
\begin{prop}
\label{prop; MatherObstruction}
Let $X\subset \PP(V)$ be a proper projective variety. Let $C_X$ to be the affine cone of $X$ in $V$. Then 
$$
Eu_{C_X}(0)=(-1)^{\dim V-1} c_{M}^X(-1) \/.
$$
Moreover, let $X^\vee$ be the dual variety of $X$ in $\PP(V^*)$, then we have
$$
(-1)^{d+d^*+\dim V}\cdot Eu_{C_X}(0)=Eu_{C_{X^\vee}}(0) \/.
$$
\end{prop}
\begin{proof}
The first argument is Proposition 3.17 in \cite{Aluffi1}. For the second part, plug $H=-1$ into the involution formula we have
\begin{align*}
Eu_{C_{X^\vee}}(0)
=& (-1)^{\dim V-1} c_M^{X^\vee}(-1) =(-1)^{\dim V-1+ d^*} \cJ_{n}((-1)^{d}c_M^X(H))(-1)\\
=& (-1)^{\dim V-1+ d^*+d}  \left(c_M^X(0)+ c_M^X(-1)\cdot (-1)^{\dim V }\right) \\
=& (-1)^{d+d^*+\dim V}\cdot Eu_{C_X}(0)
\end{align*}
\end{proof}

And we have the following algebraic Brasselet-Lê-Seade type formula.
\begin{prop}
\label{prop; generalizedBLS}
Let $X \subset\PP(V)$ be a projective variety, and let $X=\cup_i S_i$ be a stratification. Let $\Sigma\subset V$ and $V_i$ be the affine cones of $X$ and $S_i$ respectively.  
Recall that $\sm_{S_i}:=(-1)^{\dim V-1}c_{sm}^{S_i}(-1)$. 
Then we have 
\[
Eu_{\Sigma}(0)=\sum_{i} Eu_{X}(S_i)\cdot \sm_{S_i} \/.
\]
\end{prop}
\begin{proof}
This follows directly from the definition of the natural transformation $c_*$:
\begin{align*}
Eu_{\Sigma}(0)=& (-1)^{\dim V-1} c_M^{X}(-1) \\
=& (-1)^{\dim V-1} \sum_{i} Eu_X(S_i)  c_{sm}^{S_i}(-1) \/.
\end{align*}
\end{proof}
More generally, let $\alpha$ be a constructible function on $\Sigma$, i.e., $\alpha=\sum_i a_i Eu_{\bar{V_i}}$. Then from previous Proposition we have
\[
\alpha(0)=\sum_{i} a_i\cdot \sm_{S_i} \/.
\]
As we will see, this is exactly \cite[Theorem 0.2]{Sch02} when restricted to base field $\CC$.

In the analytic setting, let $V$ be a $\CC$-vector space of dimension $n$.
Let $X\subset \PP(V)$ be a closed subvariety with a Whitney stratification $\cup_{i=1}^m S_i$. Let $V_i\subset V$ be the affine cones of  $S_i$, and let $C(X)$ be the affine cone over $X$. For each $V_i$ we consider the space $L_{l.t}(V_i):=V_i\cap B_\epsilon\cap l^{-1}(t)$. Here $l$ is a generic linear function $V \to \CC$, $B_\epsilon$ is the small ball centered at $0$ and $t$ is sufficiently close to $0$. The spaces $L_{l.t}(V_i)$ are homotopic equivalent for generic $l$ and $t$ small enough, and are
called the link space of $V_i$ to $0$. We show that the Euler characteristics of the link spaces are exactly  
our $c_{sm}$ invariants $\sm_{V_i}$.
\begin{prop}
\label{prop; BLS}
When the base field $k$ is $\CC$, 
for any $i\leq m$ we have
\[
\sm_{V_i} = \chi (V_i\cap B_\epsilon\cap l^{-1}(t))   \/.
\] 
In particular, we build the following affine-projective connection:
$$
\chi (C(X)\cap B_\epsilon\cap l^{-1}(t))= \chi(X)-\chi(X\cap H)
$$
Here $H$ is a generic hyperplane.
\end{prop}
\begin{proof}
Note that for any $k\leq m$, the decomposition $\bar{V}_k=\cup_{i\geq k} V_i\cup \{0\}$ is a Whitney stratification of $\bar{V}_k$ such that $0\in \bar{V}_i$ for every $k$. 
We recall the   famous Brasselet-Lê-Seade formula in \cite{BLS}: 
\[
Eu_{\bar{V}_k}(0)=\sum_{i\geq k}  \chi (V_i\cap B_\epsilon\cap l^{-1}(t)) \cdot Eu_{\bar{V}_k}(V_i)  \/.
\]

Combine with previous corollary we then have
\begin{align*}
\sum_{i\geq k}^{m} Eu_{\bar{S}_k}(S_i) \cdot c_{sm}^{S_i}(-1)
=&  c_M^{\bar{S}_k}(-1) = (-1)^{n-1} Eu_{\bar{V}_k}(0) \\
=& (-1)^{n-1} \sum_{i\geq k}^{m}  \chi(V_i\cap B_\epsilon\cap l^{-1}(t)) \cdot Eu_{\bar{V}_k}(V_i) \/.
\end{align*}
For any point $x\in S_i$, notice that locally we can view $V_k$ as the product $ S_k\times \CC^*$. Thus the product property for the local Euler obstruction shows that:
$Eu_{\bar{V}_k}(V_i)=Eu_{\bar{S}_k}(S_i)$, and  thus
$$
\sum_{i\geq k}^{m} Eu_{\bar S_k}(S_i) \cdot (-1)^{n-1} \cdot c_{sm}^{S_i}(-1)=\sum_{i\geq k}^{m}  \chi (V_i\cap B_\epsilon\cap l^{-1}(t)) \cdot Eu_{\bar{S}_k}(S_i) \/. 
$$
This equation holds for every $k\in \{1,2,\cdots ,m\}$, thus we have the following linear system
$$
\sum_{i\geq k}^{m} Eu_{\bar{S}_k}(S_i) \cdot  \zeta_i =0;\quad k=1,2,\cdots m
$$
has $\zeta_i=\left( 
c_{sm}^{S_i}(-1)- (-1)^{n-1} \chi (V_i\cap B_\epsilon\cap l^{-1}(t))
\right) $ as a solution. However,  the matrix represents the linear system is upper triangular with $1$ on the diagonal, the solution must be the zero vector. 
This proves the first formula. The rest of the proposition follows from the inclusion-exclusion property of both $c_{sm}$ class and Euler characteristic:
\begin{align*}
 \chi (C(X)\cap B_\epsilon\cap l^{-1}(t))=& \sum_{i=1}^{m} \chi (V_i\cap B_\epsilon\cap l^{-1}(t)) 
 = \sum_{i=1}^{m} (-1)^{n-1} c_{sm}^{S_i}(-1)\\
 =&(-1)^{n-1} c_{sm}^{X}(-1) = \chi(X)-\chi(X\cap H) \/.
\end{align*}
The last equality comes from Proposition~\ref{prop; csm(-1)}
\end{proof}

\subsection{A Segre Computation}
We close this section with a  technical Lemma that will be used in later computations.
\begin{lemm}
\label{Lemma; Segre}
Let $E$ be a rank $e$ vector bundle on $X$, and let $\cL=\cO(1)$ be the tautological line bundle on $\PP=\PP(E)$. Let $p$ be the projection map from $\PP(E)$ to $X$.
Then we have
$$
p_*\left( \frac{c(E\otimes \cL)}{c(\cL)}\cap [\PP(E)] \right) = [X]- \frac{c_e(E^\vee)\cap [X]}{ c(E^\vee)}    \/.
$$
\end{lemm}
\begin{proof}
Following the tensor formula \cite[Example 3.2.2]{INT} we have
\begin{align*}
c(E\otimes \cL)=& \sum_{k=0}^e \left(\sum_{i=0}^k \binom{e-i}{k-i} c_i(E)\cdot c_1(\cL)^{k-i}\right) \\
=& \sum_{k=0}^{e} \left(\sum_{i=0}^{e-k} \binom{e-k}{i} c_1(\cL)^i  \right)\cdot c_k(E) \\
=& \sum_{k=0}^{e}  c_k(E)\cdot (1+c_1(\cL))^{e-k}\\
=& \sum_{k=0}^{e}  c_{e-k}(E)\cdot (1+c_1(\cL))^{k} \/;
\end{align*}
Thus we have
\begin{align*}
\frac{c(E\otimes \cL)}{c(\cL)}=&  \frac{\sum_{k=0}^{e}  c_{e-k}(E)\cdot (1+c_1(\cL))^{k}}{1+c_1(\cL)} \\
=& \frac{c_e(E)}{1+c_1(\cL)}+ \sum_{k=1}^{e} c_{e-k}(E)\cdot (1+c_1(\cL))^{k-1} \/; 
\end{align*}
Recall the definition of Segre class: 
$$
p_*(c_1(\cL)^{e+i}\cap [\PP(E)])=s_{i+1}(E)\cap [X] \/.
$$
The pushforward then becomes
\begin{align*}
p_*\left( \frac{c(E\otimes \cL)}{c(\cL)}\cap [\PP(E)] \right)
=& p_*(\frac{c_e(E)}{1+c_1(\cL)}\cap [\PP(E)] )+  s_0(E)\cap [X] \\
=& (-1)^{e-1}s(E^\vee)c_e(E)\cap [X]+[X] \\
=& [X]-s(E^\vee)c_e(E^\vee)\cap [X] \/.
\end{align*}
\end{proof}

\section{Local Euler Obstruction of Recursive Group Orbits}
\label{S; EulerGeneral}
Now we consider the following situation: 
for each $n\geq 1$, the group $G_n$ is a connected linear algebraic group and $G_n$ acts algebraically on  a vector space $V_n$ of dimension $l_n$. 
Assume that for every such $n$ there are exactly $n$ orbits. 
\begin{assu}
\label{assum; orbits}
We assume the following
\begin{enumerate}
\item All actions contain sub-actions by $k^*$ multiplications. Thus the orbits are necessarily cones.
\item The orbits can be labeled by  $\cO_{n,0}, \cO_{n,1}, \cdots ,\cO_{n,n}$  such that $\cO_{n,j}\subset \bar{\cO}_{n,i}$ whenever $j>i$. Thus $\cO_{n,0}$ is the largest orbit and $\cO_{n,n}=\{0\}$. 
\item The largest strata $\cO_{n,0}$ is dense in $V_n$.
\item For every $n$, $k$, and for any point $p\in \cO_{n,r}$, the transverse normal slice pair $(N_p\cap \cO_{n,k}, p)$ is isomorphic to $(\cO_{r,k}, 0)$. Thus we have local product structure $N_p\times \cO_{r,k}\sim \cO_{n,k}$. 
\end{enumerate}
\end{assu}
\begin{exam}
Let $G_n=GL_n\times GL_n$ act on  $V_n=k^{n}\otimes k^n$, the space of $n\times n$ matrices by $A\cdot X\cdot B$. The orbits are matrices of fixed corank $k$, denoted by $\Sigma_{n,k}$. 
Then  for each point $p$ in $\Sigma_{n,j}^\circ \subset \Sigma_{n,k}$, the normal slice $N_p$ intersect with $\Sigma_{n,k}$ at $p$. Moreover we have $N_p\cap \Sigma_{n,k}=\Sigma_{j,k}$, and  isomorphism $\Sigma_{j,k}\times k^N \cong \Sigma_{n,k}$. Here $k^N$ is centered at $p$ of complementary dimension. Similar arguments for symmetric matrices also give such sequence of actions.

For skew-symmetric matrices, we consider the sequence of $G_n=GL_{2n}$ actions on $V_n=\wedge^2 k^{2n}$ and the sequence of $G_n=GL_{2n+1}$ action on $V_n=\wedge^2 k^{2n+1}$ separately. Each of them form a recursive group actions.
\end{exam}

\begin{rema}
We can loose the assumption as follows. In fact we don't need all the orbits to match the assumption, among all the orbits we only concentrate a subset, i.e., a flag of them. Thus we can only assume the following:
For each $n$, there is a maximal flag of  orbits  labeled by  $\cO_{n,0}, \cO_{n,1}, \cdots ,\cO_{n,n}$  such that $\cO_{n,j}\subset \bar{\cO}_{n,i}$ whenever $j>i$, and $\bar{\cO}_{n,0}=\cup_{i=0}^n \cO_{n,i} $. Here  $\cO_{n,0}$ is the largest orbit. We only require the sequence of flags to satisfy the rest of Assumption~\ref{assum; orbits}, except that we want $\cO_{n,0}$ to be dense in its closure, not $V$. 
\end{rema}

\begin{theo}[Local Euler Obstructions and $c_{sm}$ Invariants]
\label{theo; EulerObstruction}
Assume that we have a sequence of group actions as in Assumption~\ref{assum; orbits}. For any $1\leq k\leq n-1$ we denote $S_{n,k}=\PP(\cO_{n,k})\subset \PP(V_n)$ to be the projectivizations. Denote the $c_{sm}$ invariants of $\cO_{n,k}$ (or $S_{n,k}$) by $\sm_{n,i}$, i.e., 
\[
\sm_{n,i}:=(-1)^{l_n-1} c_{sm}^{S_{n,i}}(-1)
\]
Then we have the following equivalence:
\[
\{\sm_{r,k}|k,r\} 
 \xlongleftrightarrow{\text{equivalent}} 
	 \{e_{k,r}|k,r\}
\/.
\]
Here by equivalence we mean that they can be derived from each other. 
Moreover, the local Euler obstructions $Eu_{\cO_{n,r}}(\cO_{n,k})$ can be computed as
\[
Eu_{\bar{\cO}_{n,k}}(\cO_{n,r})=Eu_{\cO_{r,k}}(0)=\sum_{(\mu_1>\mu_2>\cdots >\mu_l\geq k)} \sm_{n\mu_1}\cdot \sm_{\mu_1\mu_2}\cdots\cdot \sm_{\mu_l k}
\]
Here the sum is over all   (partial and full) flags $(\mu_1>\mu_2>\cdots >\mu_l)$ such that $\mu_i-1=\mu_{i+1}$ for every $i$ and $\mu_l\geq k$.
\end{theo}

\begin{proof}
Recall the product property of local Euler obstructions, then the normal slice assumption
(Assumption I $(4)$) shows that  
$Eu_{\bar{\cO}_{n,k}}(\cO_{n,r})=Eu_{\bar{\cO}_{r,k}} (\cO_{r,r})$   is independent of $n$.  We denote this number by $e_{k,r}$. 

Since $G_n$ acts linearly on $V_n$, the orbits $V_n=\cup_{i=0}^{n} \cO_{n,i} $ form a stratification of $V_n$  such that $0\in \bar{\cO}_{n,k}$ for every $k$.
Proposition~\ref{prop; MatherObstruction} shows that for any $k=1, \cdots , n-1$ we have
\[
e_{k,n}:=Eu_{\bar{\cO}_{n,k}}(0)=(-1)^{l_n-1}c_M^{\bar{S}_{n,k}}(-1) =(-1)^{l_n-1}\sum_{i=k}^{n-1} Eu_{\bar{S}_{n,k}}(S_{n,i})c_{sm}^{S_{n,i}}(-1)= \sum_{i=k}^{n-1} \sm_{n,i}\cdot e_{ki} 
\]
This formula applies to all $n$ and $k$,  thus we obtain a recursive formula. 
On the initial case $k=n$, notice that $\bar{\cO}_{k,k}=\{0\}$ and thus
$
e_{kk}=Eu_{\bar{\cO}_{k,k}}(0)=1 .
$
Then we have
\begin{align*}
e_{k,n}=& \sum_{i=k}^{n-1} \sm_{n,i}\cdot e_{ki} 
= \sum_{i=k}^{n-1} \sm_{n,i} \sum_{j=k}^{i-1} \sm_{i,j} e_{ij} \\
=& \quad \cdots \cdots\\
=& \sum_{(\mu_1>\mu_2>\cdots >\mu_l\geq k)} \sm_{m,\mu_1}\cdot \sm_{\mu_1,\mu_2}\cdots\cdot \sm_{\mu_l, k}
\end{align*}
The sum is over all the flags $(\mu_1>\mu_2>\cdots >\mu_l)$ such that $\mu_i-1=\mu_{i+1}$ for every $i$ and $\mu_l\geq k$.

Now we show that the local Euler obstructions recover the $c_{sm}$ invariants.  By definition we have  
\begin{align*}
\sm_{n,k}=& (-1)^{l_n-1}c_{sm}^{S_{n,k}}(-1)=(-1)^{l_n-1} \sum_{r=k}^{n-1} a^n_r c_{M}^{S_{n,r}}(-1) \\
=&\sum_{r=k}^{n-1} a^n_r Eu_{\bar{\cO}_{r,k}}(0)=\sum_{r=k}^{n-1} a^n_r e_{k,r} \/.
\end{align*}
Here the $a^n_r$ are the base change coefficients between $\ind_{\PP(\cO_{n,r})}$ and $Eu_{\PP(\cO_{n,r})}$, which are completely determined by $\{e_{r,n}\}$. This completes the proof.
\end{proof}

\begin{rema}
When the base field is $\CC$, this actually correspond to \cite[Definition 4,1,36]{Dimca} via Proposition~\ref{prop; BLS}. The fact that the recursive definition using Kashiwara's local Euler characteristics is equivalent to MacPherson's orbstruction theoretic definition are equivalent was proved in \cite{BDK}. Thus this theorem generalizes the equivalence 
from $\CC$ to arbitrary algebraically closed field of characteristic $0$, using Gonz\'alez-Sprinberg's algebraic formula for local Euler obstructions and our $c_{sm}$ invariants for local Euler characteristics.
\end{rema}

\section{Sectional Euler Characteristic of Group Orbits}
\label{S; SecEulerGeneral}
In this section we consider the following situation. 
\begin{assu}
\label{assum; orbits2}
Let $G$ be a connected linear algebraic group acting on $V$ with finitely many orbits. We label them by  $\cO_0, \cO_1, \cdots ,\cO_n$ such that $\cO_i\not \subset \bar{\cO}_j$ whenever $j>i$. As shown in \cite{Tevelev}, the dual action of $G$ induces $n$ orbits in $V^*$. 
We can label them such that  $\PP(\bar{\cO}_i)^\vee=\PP(\bar\cO'_{n-i})$ for $i\geq 1$.  Then we also have $\cO'_j\not \subset \bar{\cO'}_i$ whenever $i>j$. Thus the dense open orbits by $\cO_0$ and $\cO'_0$, and we have $\cO_n=\cO'_n=\{0\}$. Let $d_i$ and $d'_i$ denote the dimensions of $\cO_i$ and $\cO'_i$ respectively. We have $d_0=d'_0=\dim V=N$.
\end{assu}

\begin{defi}
Let $S_i$ and $S'_i$ be the projectivizations of $\cO_i$ and $\cO'_i$. 
For any   $i,j$ in $\{1,\cdots ,n-1\}$   we define the following numerical indices
$$
e_{i,j}:=Eu_{\bar{S}_i}(p);\quad  \chi_{i,j}:=\chi(S_i\cap L_j)
$$
for any point $p\in S_j$, and $L_j\in \cO'_j$. 
Symmetrically, we define 
$$
e'_{i,j}:=Eu_{\bar{S}'_i}(p);\quad  \chi'_{i,j}:=\chi(S'_i\cap L'_j)
$$
for any point $p\in S'_j$, and $L'_j\in S_j$.
 Let $E=[e_{i,j}]$ and $E'=[e'_{i,j}]$ be the $n-1\times n-1$ matrices.
\end{defi}

Knowing the local Euler obstructions of the orbits  lead us to the information of the sectional Euler characteristics of the dual orbits: 
\begin{theo}
\label{theo; sectionalEuler}
Assume that a $G$ representation $V$ satisfies Assumption~\ref{assum; orbits2}. 
Let $A=[\alpha_{i,j}]_{n-1\times n-1}$ and $A'=[\alpha'_{i,j}]_{n-1\times n-1}$  be the inverse matrix of  $E$ and $E'$ respectively, then
we have
\begin{align*}
\chi_{k,r}=&
\sum_{i=r}^{n-1}\alpha_{k,i}\cdot (-1)^{d_i+d'_{n-i}+N} e'_{n-i,r}+\chi(S_k)-\sm(S_k) \/; \\
\chi'_{k,r}=&
\sum_{i=r}^{n-1}\alpha'_{k,i}\cdot (-1)^{d_i+d'_{n-i}+N} e_{n-i,r}+\chi(S'_k)-\sm(S'_k) \/.
\end{align*}
\end{theo}
\begin{proof}
Notice that the matrix $E=[e_{i,j}]$ is upper triangular with $1$ on the diagonal, thus it is invertible with inverse matrix $A$ also being upper triangular. This shows that $\alpha_{i,j}=0$ whenever $i>j$. 
As pointed out in Corollary~\ref{coro; generic}, for any $S_i$ we have $\chi_{i,0}=\chi(S_i \cap H)$ and $\chi'_{i,0}=\chi(S'_i \cap H)$ equal the generic sectional Euler characteristic.
Thus recall Theorem~\ref{theo: Radon},   for any $l_s\in S'_s$  we have
$$
\sum_{i\geq k}^{n-1} e_{ki}(\chi_{is}-\chi_{i0}) =  (-1)^{d_k+d'_{n-k}+N} e'_{ks}
$$
Thus written in matrix form we have
\[
 E
\cdot 
\left(
\begin{array}{cccc}
\chi_{1,1}-\chi_{1,0} & \chi_{12}-\chi_{1,0} & \cdots & \chi_{1,n-1}-\chi_{1,0}   \\
\chi_{2,1}-\chi_{2,0} & \chi_{22}-\chi_{2,0} &\cdots  & \chi_{2,n-1}-\chi_{2,0}  \\
\cdots & \cdots & \cdots  & \cdots  \\
\chi_{n-1,1}-\chi_{n-1,0} & \chi_{n-1,2}-\chi_{n-1,0} & \cdots & \chi_{n-1,n-1} -\chi_{n-1,0}
\end{array}
\right)
=
\left(
\begin{array}{cccc}
0 & 0 & \cdots &  \hat{e}'_{n-1,n-1}  \\
0 & \cdots  &  \hat{e}'_{n-2,n-2} &   \hat{e}'_{n-2,n-1}  \\
\cdots & \cdots & \cdots  & \cdots  \\
\hat{e}'_{1,1}   & \hat{e}'_{1,2} & \cdots & \hat{e}'_{1,n-1}  
\end{array}
\right) . 
\]  
Here $\hat{e}'_{i,j}=(-1)^{d_i+d'_{n-i}+N}e'_{i,j}$.
Thus we have
\[
\chi_{k,r}-\chi_{k,0}=\sum_{i=r}^{n-1}\alpha_{k,i}\cdot (-1)^{d_i+d'_{n-i}+N}e'_{n-i,r} \/.
\]
Recall from Proposition~\ref{prop; csm(-1)} shows $\chi_{k,0}=\chi(S_k\cap H)=\chi(S_k)-\sm(S_k)$. Then we have
\begin{align*}
\chi_{k,r}=&
\sum_{i=r}^{n-1}\alpha_{k,i}\cdot (-1)^{d_i+d'_{n-i}+N} e'_{n-i,r}+\chi(S_k)-\sm(S_k) \/.
\end{align*}
The other formula is obtained directly by symmetry.
\end{proof}

In fact, as shown in the following proposition, the dual sectional Euler characteristics in $\PP(V^*)$ is principally the same with the ones in $\PP(V)$. 
\begin{prop}
The following two matrices are inverse to each other.
\[
\left(
\begin{array}{ccc}
\chi_{1,1}-\chi_{1,0} &   \cdots & \chi_{1,n-1}-\chi_{1,0}   \\
\chi_{2,1}-\chi_{2,0} &\cdots  & \chi_{2,n-1}-\chi_{2,0}  \\
\cdots & \cdots & \cdots  \\
\chi_{n-1,1}-\chi_{n-1,0} & \cdots & \chi_{n-1,n-1} -\chi_{n-1,0}
\end{array}
\right)^{-1}
=
\left(
\begin{array}{ccc}
\chi'_{1,1}-\chi'_{1,0}  & \cdots & \chi'_{1,n-1}-\chi'_{1,0}   \\
\chi'_{2,1}-\chi'_{2,0} &\cdots  & \chi'_{2,n-1}-\chi'_{2,0}  \\
\cdots & \cdots  & \cdots  \\
\chi'_{n-1,1}-\chi'_{n-1,0}  & \cdots & \chi'_{n-1,n-1} -\chi'_{n-1,0}
\end{array}
\right)
\/. 
\]
\end{prop}
\begin{proof}
Set $\beta_{k,r}=\chi_{k,r}-\chi_{k,0}$ and $\beta'_{k,r}=\chi'_{k,r}-\chi'_{k,0}$, then we have
\begin{align*}
&\left(
\begin{array}{cccc}
e_{1,1} & e_{1,2} & \cdots & e_{1,n-1}   \\
0 & e_{2,2}&\cdots  & e_{2,n-1}  \\
\cdots & \cdots & \cdots & \cdots  \\
0 & 0 & \cdots & e_{n-1,n-1}
\end{array}
\right)
\cdot 
\left(
\begin{array}{ccc}
\beta_{1,1}   &   \cdots & \beta_{1,n-1}   \\
\beta_{2,1}   & \cdots  & \beta_{2,n-1}  \\
\cdots & \cdots   & \cdots  \\
\beta_{n-1,1}   & \cdots & \beta_{n-1,n-1} 
\end{array}
\right) 
\cdot 
\left(
\begin{array}{ccc}
\beta'_{1,1}   &   \cdots & \beta'_{1,n-1}   \\
\beta'_{2,1}   & \cdots  & \beta'_{2,n-1}  \\
\cdots & \cdots   & \cdots  \\
\beta'_{n-1,1}   & \cdots & \beta'_{n-1,n-1} 
\end{array}
\right) \\
&=
\left(
\begin{array}{cccc}
0 & 0 & \cdots &  \hat{e}'_{n-1,n-1}  \\
0 & \cdots  &  \hat{e}'_{n-2,n-2} &   \hat{e}'_{n-2,n-1}  \\
\cdots & \cdots & \cdots  & \cdots  \\
\hat{e}'_{1,1}   & \hat{e}'_{1,2} & \cdots & \hat{e}'_{1,n-1}  
\end{array}
\right) 
\cdot 
\left(
\begin{array}{ccc}
\beta'_{1,1}   &   \cdots & \beta'_{1,n-1}   \\
\beta'_{2,1}   & \cdots  & \beta'_{2,n-1}  \\
\cdots & \cdots   & \cdots  \\
\beta'_{n-1,1}   & \cdots & \beta'_{n-1,n-1} 
\end{array}
\right)
\end{align*}
\begin{align*}
&=
\left(
\begin{array}{ccc}
(-1)^{T_{n-1}}   &   \cdots & 0   \\
0 & (-1)^{T_{n-2}} & \cdots \\
\cdots & \cdots   & \cdots  \\
0 & \cdots & (-1)^{T_1}
\end{array}
\right)
\cdot 
\left(
\begin{array}{ccc}
0  & \cdots &  e'_{n-1,n-1}  \\
0 & \cdots  &    e'_{n-2,n-1}  \\
\cdots & \cdots   & \cdots  \\
e'_{1,1}    & \cdots & e'_{1,n-1}  
\end{array}
\right) 
\cdot 
\left(
\begin{array}{ccc}
\beta'_{1,1}   &   \cdots & \beta'_{1,n-1}   \\
\beta'_{2,1}   & \cdots  & \beta'_{2,n-1}  \\
\cdots & \cdots   & \cdots  \\
\beta'_{n-1,1}   & \cdots & \beta'_{n-1,n-1} 
\end{array}
\right) \\
&= 
\left(
\begin{array}{ccc}
(-1)^{T_{n-1}}   &   \cdots & 0   \\
0 & (-1)^{T_{n-2}} & \cdots \\
\cdots & \cdots   & \cdots  \\
0 & \cdots & (-1)^{T_1}
\end{array}
\right)
\cdot 
\left(
\begin{array}{cccc}
\hat{e}_{1,1} & \hat{e}_{1,2} & \cdots & \hat{e}_{1,n-1}   \\
0 & \hat{e}_{2,2}&\cdots  & \hat{e}_{2,n-1}  \\
\cdots & \cdots & \cdots & \cdots  \\
0 & 0 & \cdots & \hat{e}_{n-1,n-1}
\end{array}
\right) \\
&= \left(
\begin{array}{ccc}
(-1)^{T_{n-1}+T'_1}   &   \cdots & 0   \\
0 & (-1)^{T_{n-2}+T'_2} & \cdots \\
\cdots & \cdots   & \cdots  \\
0 & \cdots & (-1)^{T_1+T'_{n-1}}
\end{array}
\right)
\cdot 
\left(
\begin{array}{cccc}
e_{1,1} & e_{1,2} & \cdots & e_{1,n-1}   \\
0 & e_{2,2}&\cdots  & e_{2,n-1}  \\
\cdots & \cdots & \cdots & \cdots  \\
0 & 0 & \cdots & e_{n-1,n-1}
\end{array}
\right) 
\end{align*}
Notice that
$
T_k+T'_{n-k}=d_k+d'_{n-k}+N+d'_{n-k}+d_k+N
$
is even, thus we complete the proof.
\end{proof}

In particular, for the reflective recursive group orbits,
the $\chi$ indices of orbits $\cO_{n,k}$ are equivalent to their local Euler obstructions: they also determine the local Euler obstructions. Here by reflective we mean the following assumption:
\begin{assu}[Reflective Assumption]
\label{assu; reflective}
For any $G_n$ action on $V_n$, the dual orbits in $V_n^{*}$ also have the normal slice property. More  precisely, we assume that
for any point $p\in \cO_{n,r}$, the normal slice pair $(N_p\cap \cO'_{n,k}, p)$ is isomorphic to $(\cO'_{r,k}, 0)$.
\end{assu}

The normal slice assumptions shows that 
$
Eu_{\bar{\cO'}_{n,k}}(\cO'_{n,r})=Eu_{\bar{\cO'}_{r,k}}(\cO'_{r,r})
$
are independent of also $n$. Thus we can welly define the indices $e_{k,r}=Eu_{\bar{\cO}_{n,k}}(\cO_{n,r})$ and $e'_{k,r}=Eu_{\bar{\cO'}_{n,k}}(\cO'_{n,r})$ respectively.
Let $S_{n,k}$ and $S'_{n,k}$ be the projectivizations, and we denote their dimensions by $d_{n,k}$ and $d'_{n,k}$ respectively. We define $\chi^n_{k,r}=\chi(S_{n,k}\cap L_{n,r})$ for $l_{n,r}\in S'_{n,r}$. Let $l_n$ denotes the dimension of $V_n$. The following Lemma  shows that the local Euler obstruction of the dual orbits are the `same' with the original orbits.
\begin{lemm}[Duality]
\label{lemm. duality}
For any recursive group orbits satisfying  Assumption~\ref{assum; orbits} and Assumption~\ref{assu; reflective}, we have
\[
e'_{i,j}= \hat{e}_{j-i, j} =(-1)^{l_n+d'_{n,k}-d_{n,k}} \cdot e_{n-k,n}  \/.
\]
\end{lemm}
\begin{proof}
Recall from Proposition~\ref{prop; MatherObstruction} that, denote $l_n=\dim V_n$ we have
\begin{align*}
e'_{k,n}=& Eu_{\bar{\cO'}_{n,k}}(\cO'_{n,n})
=(-1)^{l_n-1} c_{M}^{S'_{n,k}}(-1)  \\
=& (-1)^{l_n-1} \cdot (-1)^{l_n+d'_{n,k}-d_{n,k}} c_{M}^{S_{n,n-k}}(-1) \\
=& (-1)^{l_n+d'_{n,k}-d_{n,k}} Eu_{\bar{\cO'}_{n,n-k}}(\cO_{n,n}) \\
=& (-1)^{l_n+d'_{n,k}-d_{n,k}} \cdot e_{n-k,n} =\hat{e}_{n-k,n}
\end{align*}
\end{proof}
\begin{coro}
We consider recursive group orbits satisfying  Assumption~\ref{assum; orbits} and Assumption~\ref{assu; reflective}. Let $\chi^n_{k,r}$ denote the Euler characteristic $\chi(S_{n,k}\cap L_{n,r})$ for any hyperplane $l_{n,r}\in S'_{n,r}$. Then we have the following equivalence for each $n$:  they can be deduced from each other. 
\[
\{\chi^n_{k,r}| k,r =0,1,\cdots ,n-1\}
 \xlongleftrightarrow{\text{equivalent}} 
	\{e_{i,j}|i,j=1,\cdots n\}\/.
\]
In particular, the $n$th level sectional Euler characteristics $\{\chi^n_{k,r}| k,r =0,1,\cdots ,n-1\}$ control the $n-1$th level $\{\chi^{n-1}_{k,r}| k,r =0,1,\cdots ,n-2\}$.

\end{coro}
\begin{proof}
We have shown how to obtain the sectional Euler characteristic $\chi^n_{k,r}$ from local Euler obstructions $e'_{k,r}$ and $\sm_{k,r}$. By the previous Duality Lemma and Theorem~\ref{theo; EulerObstruction}, both of them are completely determined by $e_{k,r}$. 
Thus we have proved the right-to-left arrow, and  
we just need to show that we can compute $e_{i,j}$ from $\chi^n_{k,r}$. Applying Ernstr\"om's theorem we have
\[
 \left(
\begin{array}{cccc}
e_{1,1} & e_{1,2} & \cdots &  e_{1,n-1}  \\
0 & e_{2,2} & \cdots   &  e_{2,n-1}  \\
\cdots & \cdots & \cdots  & \cdots  \\
0   & 0 & \cdots & e_{n-1,n-1}  
\end{array}
\right) 
\cdot 
[\chi^n_{k,r}-\chi^n_{k,0}]_{k,r}
=
-
\left(
\begin{array}{cccc}
0 & 0 & \cdots &  e_{0,n-1}  \\
0 & \cdots  &  e_{0,n-2} &   e_{1,n-1}  \\
\cdots & \cdots & \cdots  & \cdots  \\
e_{0,1}   & e_{1,2} & \cdots & e_{n-2,n-1}  
\end{array}
\right) . 
\]
Here the two triangular matrices are invertible, thus the matrix $[\beta^n_{k,r}:=\chi^n_{k,r}-\chi^n_{k,0}]_{k,r}$ is also invertible. Then for $1\leq k\leq n-1$ we have
\[
\left(
\begin{array}{cccc}
\beta_{1,1} & \beta_{2,1} & \cdots &  \beta_{n-1,1}  \\
\beta_{1,2} & \beta_{2,2} & \cdots   &  \beta_{n-1,2}  \\
\cdots & \cdots & \cdots  & \cdots  \\
\beta_{1,n-1} & \beta_{2,n-1} & \cdots   &  \beta_{n-1,n-1}  
\end{array}
\right) 
\cdot 
\left(
\begin{array}{c}
0  \\
\cdots   \\
0  \\
e_{k,k}\\
e_{k,k+1}  \\
\cdots \\
e_{k,n-1} 
\end{array}
\right)
=
\left(
\begin{array}{c}
0  \\
\cdots   \\
0  \\
e_{0,n-k}\\
e_{1,n-k+1}  \\
\cdots \\
e_{k-1,n-1} 
\end{array}
\right)
\]
This shows that if $\{e_{i,j}\}$ are known for $i\leq j<k$, then $e_{i,k}$ are uniquely determined by $[\beta^n_{k,r}]_{k,r}$. When $k=1$, $e_{0,j}=1$. Thus by induction on $k$ we have proved that $e_{i,j}$ can be obtained from $\chi^n_{k,r}$. 
\end{proof}

\section{Matrix Rank Loci}
\label{S; rankloci}
For the rest of the paper we discuss the matrix rank loci. 
By a matrix rank loci we mean the following three cases. Let $M_n$, $M^\wedge_n$ and $M^S_n$ be the space of $n\times n$ ordinary, skew-symmetric and symmetric matrices over $k$ respectively. We know that $M_n=(k^n)^{\otimes 2}$, $M_n^\wedge=\wedge^2 k^n$ and $M_n^S=Sym^2 k ^n$ are of dimension $n^2$, $\binom{n+1}{2}$ and $\binom{n}{2}$ respectively. The group $GL_n(k)\times GL_n(k)$ and $GL_n(k)$ act  on $M_n$, $M^\wedge_n$ and $M^S_n$ by sending a matrix $A$ to $PAQ$, $PAP^t$ and $PAP^t$. For all three cases, there are only finitely orbits consisting of matrices of the same rank. 
For any $0\leq k\leq n$, we denote the orbits as
\begin{align*}
\Sigma^\circ_{n,k}:=& \{A\in M_n| \rk A= n-k\};\\
\Sigma^{S \circ}_{n,k}:=& \{A\in M^S_n| \rk A =n-k\} \/. 
\end{align*}
Since a skew-symmetric matrix can only be of even rank, we denote
\[
\Sigma^{\wedge \circ}_{2n+1,2k+1}:=  \{A\in M^\wedge_{2n+1}| \rk A= 2n-2k\};\quad \Sigma^{\wedge \circ}_{2n,2k}:=  \{A\in M^\wedge_{2n}| \rk A= 2n-2k\}\/.
\]
We call those orbits the ordinary, symmetric and skew-symmetric rank loci. We denote $\Sigma_{n,k}$, $\Sigma^\wedge_{n,k}$ and $\Sigma^S_{n,k}$
to be there closure.

Note that all three $GL_n$  actions  contain  $k^*$ multiplications, thus the actions pass to projectived spaces $\PP(M_n)$, $\PP(M^\wedge_n)$ and $\PP(M^S_n)$. We denote the projectived orbits (and their closure) by $\tau^\circ_{n,k}$, $\tau^{\wedge \circ}_{n,k}$ and 
$\tau^{S \circ}_{n,k}$ (and  $\tau_{n,k}$, $\tau^\wedge_{n,k}$ and $\tau^S_{n,k}$) respectively.  We list here some basic properties of rank loci, for details we refer to \cite{R-Prom}, \cite{Prom}, \cite{Eisenbud} and \cite{Harris-Tu}. 
\begin{prop}
\label{prop; rankloci}
With the notations mentioned above.
\begin{enumerate}
\item The rank loci are reduced irreducible quasi-projective varieties.
\item $($Dimension$)$ $\dim \Sigma_{n,k}=n^2-k^2$, $\dim \Sigma^S_{n,k}=\frac{(n-k)(n+k+1)}{2}$ and $\dim \Sigma^\wedge_{n,k}=\frac{(n-k)(n+k-1)}{2}$ .
\item $($Singularity$)$ When $1\leq k\leq n-1$, $\Sigma_{n,k}$, $\Sigma^\wedge_{n,k}$ and $\Sigma^S_{n,k}$ are singular with singularity 
$\Sigma_{n,k+1}$, $\Sigma^\wedge_{n,k+2}$ and $\Sigma^S_{n,k+1}$ respectively. For $k=0, n$ they are smooth. 
\item $($Product Structure$)$ For any $A\in \Sigma_{n,r}$, the normal slice $N_A$ is isomorphic to $M_{r}$, and 
intersects $\Sigma_{n,r}$ at $\{A\}=0$. Moreover the intersection $\Sigma_{n,k}\cap N_A$ is isomorphic to $\Sigma_{r,k}$.  
This gives a local product structure  $(\Sigma_{n,k}, A)\cong (\Sigma_{r,k},0)\times k^{N}$. The same property holds for $\Sigma^\wedge_{n,k }$ and $\Sigma^S_{n,k}$.
\item $($Duality$)$ For $k=1,\cdots ,n-1$, 
the dual variety of  $\tau_{n,k}$, $\tau^S_{n,k}$ and $\tau^\wedge_{2n+1,2k+1}$  $(\text{or } \tau^\wedge_{2n ,2k})$  are isomorphic to $\tau_{n,n-k}$ ,  $\tau^S_{n,n-k}$  and  $\tau^\wedge_{2n+1,2n+1-2k}$ $(\text{or } \tau^\wedge_{2n ,2n-2k})$  respectively.
\item For ordinary matrix and skew-symmetric matrices, we have 
\begin{align*}
\chi(\tau^{\wedge \circ}_{n,k})= 
\begin{cases}
0 & k< n-2  \\
\binom{n}{2} & k=n-2
\end{cases} 
;\quad 
\chi(\tau^\circ_{n,k})= 
\begin{cases}
0 & k < n-1  \\
n^2  & k=n-1
\end{cases}
\end{align*}
\item For symmetric matrices, we have the following slightly different result
\begin{align*}
\chi^S_{n,k}:=\chi(\tau^{S \circ}_{n,k})
=&
\begin{cases}
0 & k< n-2 \\
\binom{n}{2} & k=n-2 \\
\binom{n}{1} & k=n-1
\end{cases}
\end{align*}
\end{enumerate}
\end{prop}

Most importantly, all three types rank loci admit  natural resolutions of singularities: the Tjurina transforms defined as  the incidence of Grassmannians.  
Let $G(k,n)$ be the Grassmannian of $k$-planes in $V=k^n$, we denote $S$ and $Q$ to be the universal sub and quotient bundles.
For ordinary rank loci $\tau_{n,k}$, the (projective) Tjurina transform is defined by
\[
\hat{\tau}_{n,k}:=\{(\Lambda, A)|\Lambda\subset \ker A\}\subset \PP(M_n)\times G(k,n) \/.
\]
The skew-symmetric and symmetric Tjurina transforms are defined as
\begin{align*}
\hat{\tau}^\wedge_{n,k}:=&\{(\Lambda, X)|X|_{\Lambda}=0\}\subset \PP(M^\wedge_n)\times G(k,n) ;\\
\hat{\tau}^S_{n,k}:=&\{(\Lambda, X)|X|_{\Lambda}=0\}\subset \PP(M^S_n)\times G(k,n) ;
\end{align*}
For all three cases, set $X=\tau$, $X=\tau^\wedge$ and $X=\tau^S$, and set $\PP^N$ by $\PP(M_n)$, $\PP(M^\wedge_n)$ and $\PP(M^S_n)$
we have commutative diagrams
\[
\begin{tikzcd}
 & \hat{X}_{n,k} \arrow{r}{} \arrow{d}{p} \arrow{dl}{q} & G(k,n) \times \PP^{N} \arrow{d} \\
G(k,n) & X_{n,k} \arrow{r}{} & \PP^{N}.
\end{tikzcd}
\]
The second projection $p$ is a resolution of singularity,  and is isomorphic over $\tau^\circ_{n,k}$. The first projections $q$ identify  
the Tjurina transforms with  projectivized bundles:
\[
\hat{\tau}_{n,k}\cong \PP(Q^{\vee n});\quad  \hat{\tau}^\wedge_{n,k}\cong \PP(\wedge^2 Q^\vee); \quad \hat{\tau}^S_{n,k}\cong \PP(Sym^2 Q^\vee) \/.
\]

\section{Local Euler Obstruction of Rank Loci}
\label{S; EulerRankLoci}
In this section we apply Theorem~\ref{theo; sectionalEuler} to the case of matrix rank loci.  
\subsection{Ordinary Matrix}
For ordinary matrix cells, the local Euler obstruction are computed over $\CC$ in \cite{NG-TG} and over arbitrary algebraically closed field in \cite{Xiping2}. In  \cite{NG-TG} the authors used  recursive method  based on the knowledge of $\sm_{n,i}$ computed in \cite{E-G} via topology method. In \cite{Xiping2} the formula came from direct  computation via the knowledge of Nash blowup and the Nash bundle. Here we propose a proof to the formula for $\sm_{n,i}$ different with \cite{E-G} and \cite{Xiping2}. The proof is  algebraic,  thus works for    arbitrary algebraically closed field of characteristic $0$.
\begin{prop}
\label{prop; ordsm}
For ordinary matrix rank loci we have
\[
\sm_{n,i}=(-1)^{n+1-k}\binom{n}{k} \/.
\]
\end{prop}
\begin{proof}
Recall the Tjurina diagram
\[
\begin{tikzcd}
 & \hat{\tau}_{n,k} \arrow{r}{} \arrow{d}{p} \arrow{dl}{q} & G(k,n) \times \PP^{N} \arrow{d} \\
G(k,n) & \tau_{n,k} \arrow{r}{} & \PP^{N}.
\end{tikzcd}
\]
In \cite{Xiping1} we proved the following  formula
\[
c_{sm}^{\tau_{n,k}^\circ}(H)=\sum_{i=k}^{n-1 } (-1)^{i-k}\binom{i}{k} p_*c_{sm}^{\hat{\tau}_{n, i}} (H).
\]
Thus it amounts to compute $p_*c_{sm}^{\hat{\tau}_{n, i}}(-1)$.  
Recall that $\hat{\tau}_{n,k}$ is isomorphic to the projectivized bundle $\PP(Q^{\vee n})$. 
Let $S$ and $Q$ be the universal sub and quotient bundles on Grassmannian $G(k,n)$, and let $\cO(1)$ be the tautological line bundle. We have
\[
(-1)^{n^2-1}p_*c_{sm}^{\hat{\tau}_{n,k}}(-1)=\int_{\PP(Q^{\vee n})} \frac{c(S^\vee\otimes Q)c(Q^{\vee n}\otimes \cO(1))}{c(\cO(1))} \cap [\PP(Q^{\vee n})]
\]
By Lemma~\ref{Lemma; Segre} we have
\begin{align*}
(-1)^{n^2-1}p_*c_{sm}^{\hat{\tau}_{n,k}}(-1)
=&\int_{\PP(Q^{\vee n})} \frac{c(S^\vee\otimes Q)c(Q^{\vee n}\otimes \cO(1))}{c(\cO(1))} \cap [\PP(Q^{\vee n})]\\
=& \int_{G(k,n)} \left( 1-\frac{c_{n(n-k)}(Q^{ n}))}{c(Q^{ n}))} \right)\cap [G(k,n)] \\
=& \int_{G(k,n)}c(T_{G(k,n)})\cap  [G(k,n)]=\binom{n}{k}
\end{align*}

Thus we have  
\begin{align*}
\sm_{n,k}=& (-1)^{n^2-1} c_{sm}^{\tau_{n,k}^\circ}(-1)
=\sum_{i=k}^{n-1 } (-1)^{i-k}\binom{i}{k} (-1)^{n^2-1} p_*c_{sm}^{\hat{\tau}_{n, i}} (-1) \\
=& \sum_{i=k}^{n-1} (-1)^{i-k} \binom{n}{i} \binom{i}{k}
= \binom{n}{k} \cdot \sum_{i=k}^{n-1}(-1)^{i-k}\binom{n-k}{i-k} \\
=& (-1)^{n+1-k} \binom{n}{k}
\end{align*}
\end{proof}	

Follow from Theorem~\ref{theo; EulerObstruction} we have:
\begin{theo}
\label{theo; OrdEuler}
The local Euler obstruction function of $\tau_{n,k}$ is
\[
e_{i,j}=Eu_{\tau_{n,i}}(\tau_{n,j}^\circ)=\binom{j}{i} \/.
\]
\end{theo} 
 \begin{proof}
We prove by inductions. For the initial case, we have $e_{i,i}=1$ for any $i$.
Assume that  $e_{i,j}=\binom{j}{i}$ for all $i\leq j<n$. Then Proposition~\ref{prop; csm(-1)} we have
\begin{align*}
 e_{k,n}
 =& \sum_{r=k}^{n-1} e_{k,r}\cdot  \sm_{n,r} = \sum_{r=k}^{n-1} \binom{r}{k} (-1)^{n+1-r}\binom{n}{r} \\
 =& \binom{n}{k} \cdot (-1)\cdot \sum_{j=0}^{n-k-1}(-1)^{n-k-j}\binom{n-k}{j} \\
 =& \binom{n}{k} \cdot (-1)\cdot (0^{n-k}-(-1)^{(n-k)-(n-k)}) 
 =\binom{n}{k}
\end{align*}
\end{proof}

\subsection{Skew-Symmetric Matrix}
For skew-symmetric rank loci we have $V_n=\wedge^2 k^n$ and $G_n=GL_n(k)$. 
we denote to   orbits of  of rank $n-k$ matrices in $\wedge^2 k^n$ and $\PP(\wedge^2 k^n)$ by $\Sigma^{\wedge \circ}_{n,k}$ and $\tau^{\wedge \circ}_{n,k}$ respectively.  
Here $n-k=0,2,4,\cdots  \lfloor \frac{n}{2} \rfloor$.
Let $\Sigma_{n,k}$ and $\tau_{n,k}$ be their closures. The local product structure shows
that, for $m\geq n$ we have
\[
Eu_{\tau^\wedge_{n,k}}(\tau^{\wedge \circ}_{n,k+j})=Eu_{\tau^\wedge_{m,k}}(\tau^{\wedge \circ}_{m,k+j})  \/.
\]

For complex skew-symmetric rank loci, the local Euler obstruction are computed in \cite[Chapter 9]{Prom}, in which  the author computed $\sm_{m,i}$ using topology method. Here we propose an algebraic formula which works for general $k$ of characteristic $0$. 
\begin{prop}
\label{prop; smSkew}
The $c_{sm}$ invariants are given by
\[
\sm_{2n,2k}=\sm_{2n+1,2k+1} =(-1)^{n-k+1}\binom{n}{k} \/.
\]
\end{prop}
\begin{proof}
Recall the Tjurina diagram for skew-symmetric rank loci. 
\[
\begin{tikzcd}
 & \hat{\tau}^\wedge_{m,k} \arrow{r}{} \arrow{d}{p} \arrow{dl}{q} & G(k,n) \times \PP(\wedge^2 k^m) \arrow{d} \\
G(k,n) & \tau^\wedge_{m,k} \arrow{r}{} & \PP(\wedge^2 k^m).
\end{tikzcd}
\]
Here $\hat{\tau}^\wedge_{m,k}$ is isomorphic to the projective bundle $\PP(\wedge^2 Q^\vee)$. Thus we have
\[
(-1)^{\binom{m}{2}-1} p_*(c_{sm}^{\hat{\tau}^\wedge_{m,k}})(-1)=\int_{\PP(\wedge^2 Q^\vee)} \frac{c(S^\vee\otimes Q)c(\wedge^2 Q^\vee)\otimes \cO(1))}{c(\cO(1))} \cap 
[\PP(\wedge^2 Q^\vee)]
\]
By Lemma~\ref{Lemma; Segre} we have
\begin{align*}
(-1)^{\binom{m}{2}-1} p_*(c_{sm}^{\hat{\tau}^\wedge_{m,k}})(-1)
=& \int_{G(k,m)} c(S^\vee\otimes Q)\cdot \left(1-\frac{c_{top}(\wedge^2 Q)}{c(\wedge^2 Q)}   \right) \cap [G(k,m)]\\
=& \binom{m}{k}-\int_{G(k,m)} \frac{c(S^\vee\otimes Q)c_{top}(\wedge^2 Q)}{c(\wedge^2 Q)} \cap [G(k,m)] 
\end{align*}

For any $A\in  \tau^{\wedge \circ}_{m,r}$, the fiber $p^{-1}(A)$ is isomorphic to the Grassmannian $G(k,r)$. Notice that $A$ can only have even degree, we need to consider $m=2n$ and $m=2n+1$ cases separately.
First we consider the even case. The pushforward of constructible functions shows that
\begin{align*}
(-1)^{\binom{2n}{2}-1} p_*(c_{sm}^{\hat{\tau}^\wedge_{2n,2k}})(-1)
=&\sum_{r=k}^{n-1} \binom{2r}{2k} (-1)^{\binom{2n}{2}-1}c_{sm}^{\tau^{\wedge\circ}_{2n,2r}}(-1) \\
=& \sum_{r=k}^{n-1} \binom{2r}{2k} \sm_{2n,2r} \/.
\end{align*}
When $m=2n+1$, we have
\[
(-1)^{\binom{2n+1}{2}-1} p_*(c_{sm}^{\hat{\tau}^\wedge_{2n+1,2k+1}})(-1)=\sum_{r=k}^{n-1} \binom{2r+1}{2k+1} \sm_{2n+1,2r+1} \/.
\]
Notice that the matrices $[\binom{2r}{2k}]_{1\leq r,k\leq n-1}$ and $[\binom{2r+1}{2k+1}]_{1\leq r,k\leq n-1}$ are invertible, thus it amounts to prove the following Schubert identities:
\begin{align*}
\int_{G(2k,2n)} \frac{c(S^\vee\otimes Q)c_{\binom{2n-2k}{2}}(\wedge^2 Q)}{c(\wedge^2 Q)} \cap [G(2k,2n)] 
=&  \sum_{r=k}^{n} (-1)^{n-k} \binom{2r}{2k}\binom{n}{r} 
\\
\int_{G(2k+1,2n+1)} \frac{c(S^\vee\otimes Q)c_{\binom{2n-2k}{2}}(\wedge^2 Q)}{c(\wedge^2 Q)} \cap [G(2k+1,2n+1)] 
=& \sum_{r=k}^{n} (-1)^{n-k} \binom{2r+1}{2k+1}\binom{n}{r} 
\end{align*}
When the base field is $\CC$, this was proved in \cite[Chapter 9]{Prom}. Since the equations are nothing but binomial identities, thus they hold for arbitrary field $k$ of charcteristic $0$.
\end{proof}

Mimicking the proof for the ordinary rank loci case we have:
\begin{theo}
\label{theo; SkewEuler}
The local Euler obstructions of skew-symmetric rank loci are
\[
Eu_{\tau^\wedge_{2n,2k}}(\tau^{\wedge \circ}_{2n,2r})=Eu_{\tau^\wedge_{2n+1,2k+1}}(\tau^{\wedge\circ}_{2n+1,2r+1})=\binom{r}{k} \/.
\]
\end{theo}

\subsection{Symmetric Matrix} 
Unlike the ordinary and skew-symmetric case, the Euler obstructions of complex symmetric rank loci are unknown. Thus we compute them here.

In this section we have $V_n=Sym^2 k^n$ to be the space of all symmetric $n\times n$ matrices, and $G_n=GL_n(k)$ acts on $V_n$
by $PAP^t$. The group orbits consist of  symmetric matrices of rank $n-k$ for $k=0,1,\cdots n-1$, and  are
denoted by 
$\Sigma^{S \circ}_{n,k}$. Let $\tau^{S \circ}_{n,k}$  be their projectivizations in  $\PP(Sym^2 k^n)$.
Let $\Sigma^S_{n,k}$ and $\tau^S_{n,k}$ be their closures. The  product structure shows that
\[
Eu_{\Sigma^S_{n,k}}(\Sigma^{S \circ}_{n,k+j})=Eu_{\Sigma^S_{m,k}}(\Sigma^{S \circ}_{m,k+j})=Eu_{\tau^S_{n,k}}(\tau^{S \circ}_{n,k+j}) \/.
\]
for $k+j<n$. First we compute the $c_{sm}$ indices $\sm_{n,k}$.
\begin{prop}
\label{prop; smSym}
For any $1\leq k\leq m-1$, we have the following formula
\begin{align*}
\sm_{m,k}:=(-1)^{\binom{m+1}{2}-1} \cdot c_{sm}^{\tau^{S \circ}_{m,k} }(-1)=
\begin{cases}
(-1)^{n-i+1} \cdot \binom{n}{i} & k=2i, m=2n\\
0  &  k=2i+1, m=2n \\
(-1)^{n-i} \binom{n}{i} &   k=2i, m=2n+1\\
(-1)^{n-i+1} \binom{n}{i} & k=2i+1, m=2n+1
\end{cases}
\end{align*}
\end{prop}
\begin{proof}
Consider the Tjurina transform diagram  of symmetric rank loci:
\[
\begin{tikzcd}
 & \hat{\tau}^S_{n,k} \arrow{r}{} \arrow{d}{p} \arrow{dl}{q} & G(k,n) \times \PP(Sym^2 k^n) \arrow{d} \\
G(k,n) & \tau^S_{n,k} \arrow{r}{} & \PP(Sym^2 k^n) \/.
\end{tikzcd}
\]
Here $\hat{\tau}^S_{n,k}$ is isomorphic to the projective bundle $\PP(Sym^2 Q^\vee)$ over $G(k,n)$, for $Q$ being the universal quotient bundle.
Notice that for any $A\in \tau^{S \circ}_{m,k+j}$, its fiber $p^{-1}(A)$ is isomorphic to $G(k,k+j)$. Thus we have
\[
p_*c_{sm}^{\hat{\tau}^S_{m,k}}(H)=\sum_{j=k}^{m-1} \binom{j}{k} c_{sm}^{\hat{\tau}^{S \circ}_{m,j} }(H)
\]
Following the same argument in ordinary case we have
\[
c_{sm}^{\tau^{S \circ}_{m,k}}(H)= \sum_{j=k}^{m-1} (-1)^{j-k}\binom{j}{k} p_*c_{sm}^{\hat{\tau}^S_{m,j}} (H).
\]
Thus it amounts to compute the indices $p_*c_{sm}^{\hat{\tau}^S_{m,k+i}}(-1)$. Since $\hat{\tau}^S_{m,k}\cong \PP(Sym^2 Q^\vee)$, let $\cO(1)$ be the tautological line bundle
we have
\[
(-1)^{\binom{m+1}{2}-1}p_*c_{sm}^{\hat{\tau}^S_{m,k}}(-1)=\int_{\PP(Sym^2 Q^\vee)} \frac{c(S^\vee\otimes Q)c(Sym^2 Q^\vee)\otimes \cO(1))}{c(\cO(1))} \cap 
[\PP(Sym^2 Q^\vee)]
\]
By Lemma~\ref{Lemma; Segre} we have
\begin{align*}
(-1)^{\binom{m+1}{2}-1}p_*c_{sm}^{\hat{\tau}_{n,k}}(-1)
=&\int_{\PP(Sym^2 Q^\vee)} \frac{c(S^\vee\otimes Q)c(Sym^2 Q^\vee\otimes \cO(1))}{c(\cO(1))} \cap [\PP(Sym^2 Q^\vee)]\\
=& \int_{G(k,n)} c(S^\vee\otimes Q)\cdot \left(1-\frac{c_{top}(Sym^2 Q)}{c(Sym^2 Q)}   \right) \cap [G(k,n)]\\
=& \binom{n}{k}-\int_{G(k,n)} \frac{c(S^\vee\otimes Q)c_{top}(Sym^2 Q)}{c(Sym^2 Q)} \cap [G(k,n)]
\end{align*}

\begin{lemm}
We have the following Schubert formula.
\begin{align*}
\int_{G(k,2n)} \frac{c(S^\vee\otimes Q)c_{top}(Sym^2 Q)}{c(Sym^2 Q)}=& \sum_{r=k-n}^{n} \binom{n}{r}\binom{r}{k-n}=2^{2n-k}\cdot \binom{n}{k-n} \\
\int_{G(k,2n+1)} \frac{c(S^\vee\otimes Q)c_{top}(Sym^2 Q)}{c(Sym^2 Q)}=& \sum_{r=k-n-1}^n \binom{n}{r}\binom{r}{k-n-1}=2^{2n-k+1}\cdot \binom{n}{k-n-1} \/.
\end{align*}
\end{lemm}
\begin{proof}[Proof of Lemma]
Recall the standard duality  isomorphism
\[
\phi\colon G_1:=G(k,n)=G(k,V)\to G(n-k,n)=G(n-k, V^\vee)=:G_2 \/.
\]
For $i=1,2$ we denote $S_i, Q_i$ to be the universal sub and quotient bundles. Then $\phi^*S_2=Q_1^\vee$ and $\phi^*Q_2=S_1^\vee$. Thus we have
\[
\int_{G(k,2n)} \frac{c(S^\vee\otimes Q)c_{top}(Sym^2 Q)}{c(Sym^2 Q)}=\int_{G(2n-k,2n)} \frac{c(Q\otimes S^\vee)c_{top}(Sym^2 S^\vee)}{c(Sym^2 S^\vee)} \/.
\]
Thus it amounts to prove that 
\begin{align*}
\int_{G(2n-k,2n)} \frac{c(Q\otimes S^\vee)c_{top}(Sym^2 S^\vee)}{c(Sym^2 S^\vee)}&=2^{2n-k}\cdot \binom{n}{k-n}  \\
\int_{G(2n-k+1,2n+1)} \frac{c(Q\otimes S^\vee)c_{top}(Sym^2 S^\vee)}{c(Sym^2 S^\vee)}&=  2^{2n-k+1}\cdot \binom{n}{k-n-1} \/.
\end{align*}
We leave the proof of above Schubert identities to Corollary~\ref{coro; Schubert}, where we identify such Schubert integrations as the Euler characteristics of certain moduli spaces, and compute them using geometry method.
\end{proof}
Thus from the Lemma we have
\begin{align*}
(-1)^{\binom{m+1}{2}-1}p_*c_{sm}^{\hat{\tau}^S_{2n,k}}(-1)=& \binom{2n}{k}-     2^{2n-k}\cdot \binom{n}{k-n}  \\
(-1)^{\binom{m+1}{2}-1}p_*c_{sm}^{\hat{\tau}^S_{2n+1,k}}(-1)=&\binom{2n+1}{k}- 2^{2n-k+1}\cdot \binom{n}{k-n-1}   \/.
\end{align*}
Here we make the convention that $\binom{n}{k}=0$ whenever $k<0$.

Then for $m=2n$  we have
\begin{align*}
\sm_{2n,k}=& \sum_{j=k}^{2n-1} (-1)^{j-k}\binom{j}{k} p_*c_{sm}^{\hat{\tau}_{2n,j}} (-1) \\
=& \sum_{j=k}^{2n-1} (-1)^{j-k}\binom{j}{k} \left(\binom{2n}{j}-  2^{2n-j}\cdot \binom{n}{j-n} \right)  \\
=& 
\begin{cases}
(-1)^{n-i+1} \cdot \binom{n}{i} & k=2i\\
0  &  k=2i+1
\end{cases}
\/.
\end{align*}
 For $m=2n+1$ we have
\begin{align*}
\sm_{2n+1,k}=& \sum_{j=k}^{2n} (-1)^{j-k}\binom{j}{k} p_*c_{sm}^{\hat{\tau}_{2n+1,j}} (-1) \\
=& \sum_{j=k}^{2n} (-1)^{j-k}\binom{j}{k} \left(\binom{2n+1}{j}-  2^{2n+1-j}\cdot \binom{n}{j-n-1} \right)  \\
=& 
\begin{cases}
(-1)^{n-i} \binom{n}{i} &   k=2i\\
(-1)^{n-i+1} \binom{n}{i} & k=2i+1
\end{cases}
\/.
\end{align*}
The last steps are  binomial identities that can be proved by induction. We leave the details to the readers. 
\end{proof}

Now we   proceed to compute the local Euler obstructions. 
\begin{theo}
\label{theo; SymEuler}
Let $e_{k,n}$ be the local Euler obstruction of $\Sigma^S_{n,k}$ at $0$. Then we have
\begin{align*}
e_{k,m}=
\begin{cases}
\binom{n}{i}   &   k=2i, m=2n\\
 0 & k=2i+1, m=2n \\
\binom{n}{i}   &   k=2i, m=2n+1\\
\binom{n}{i} & k=2i+1, m=2n+1 \\
\end{cases} 
\end{align*}
\end{theo}
\begin{proof}
We  prove  by induction on $n$.
When $n=1$, Proposition~\ref{prop; generalizedBLS} shows that
\begin{align*}
e_{1,2}=& e_{1,1}\cdot \sm_{2,1} =0 \\
e_{1,3}=& e_{1,2}\cdot \sm_{3,2}+e_{1,1}\cdot \sm_{3,1}=0\cdot 1+1\cdot 1=1 =\binom{1}{0}\\
e_{2,3}=& e_{2,2}\cdot \sm_{3,2}=1\cdot 1=\binom{1}{1}
\end{align*}
Assume that the theorem holds true for all $n\leq N$ and all $k\leq 2N+1$. For $n=N+1$, combine Proposition~\ref{prop; generalizedBLS} and Proposition~\ref{prop; smSym} we have
\begin{align*}
e_{2k+1,2n}=& \sum_{j=2k+1}^{2n-1} e_{2k+1,j}\cdot \sm_{2n,j} \\
=& \sum_{j=k+1}^{n-1} e_{2k+1,2j}\cdot \sm_{2n,2j}+ \sum_{j=k}^{n-1} e_{2k+1,2j+1}\cdot \sm_{2n,2j+1}  \\
=& \sum_{j=k+1}^{n-1} 0\cdot \sm_{2n,2j}+ \sum_{j=k}^{n-1} e_{2k+1,2j+1}\cdot 0  \\
=& 0 \\
e_{2k,2n}=& \sum_{j=2k }^{2n-1} e_{2k ,j}\cdot \sm_{2n,j} \\
=& \sum_{j=k}^{n-1} e_{2k,2j}\cdot \sm_{2n,2j}+ \sum_{j=k}^{n-1} e_{2k,2j+1}\cdot \sm_{2n,2j+1}  \\
=& \sum_{j=k}^{n-1} e_{2k,2j}\cdot \sm_{2n,2j}= \sum_{j=k}^{n-1} (-1)^{j-k+1}\cdot \binom{j}{k} \binom{n}{j} \\
=& \sum_{j=k}^{n-1} (-1)^{n-j+1} \binom{j}{k}\cdot \binom{n}{j}  
= \binom{n}{k}\cdot  \sum_{j=k}^{n-1} (-1)^{j-k+1} \binom{n-k}{j-k} \\
=& \binom{n}{k}
\end{align*} 
The proof for $e_{2k,2n+1}=e_{2k+1,2n+1}=\binom{n}{k}$ are exactly the same. Thus we have proved the induction step, and hence the theorem. 
\end{proof}

\section{Sectional Euler Characteristic of Rank Loci}
\label{S; SecEulerRankLoci}
In this section we apply Theorem~\ref{theo; EulerChern} to compute the sectional Euler characteristic of matrix rank loci.  
\subsection{Ordinary Matrix}
Recall that  the dual variety of $\tau_{n,k}\subset \PP^{n^2-1}$ is exactly $\tau_{n,n-k}\subset \PP^{* n^2-1}$, with dimensions $d=\dim\tau_{n,k} = (n+k)(n-k)-1$ and 
$d^*=\dim \tau_{n,n-k} = (2n-k)k-1$. We have the following.
\begin{theo}
\label{theo; secEulerOrd}
Let $L$ be a hyperplane correspond to $l\in \tau_{n,r}^\circ\in \PP^{^* n^2-1}$. Then we have
\[
\chi(\tau_{n,k}^\circ\cap L)=
\begin{cases}
\sum_{i=k}^{n } (-1)^{i-k} \binom{i}{k}\binom{r}{n-i}   & k<n-1\\
r -n +n^2-1 & k=n-1 
\end{cases} \/.
\]
Here we make the convention that $\binom{a}{b}=0$ whenever $b<0$.
\end{theo}
\begin{proof}
Theorem~\ref{theo: Radon} shows that
\[
(Eu_{\tau_{n,k}})^\vee = Eu_{\tau_{n,n-k}}+e\ind_{\PP^{n^2-1}}
\]
Thus for any $r$ we have 
\[
\sum_{j=0}^{n-1} Eu_{\tau_{n,k}}(\tau_{n,j}^\circ)\left(\chi^n_{k,r}-\chi^n_{k,0} \right)=Eu_{\tau_{n,n-k}}(\tau_{n,r}^\circ) \/.
\]
Recall that $Eu_{\tau_{n,k}}(\tau_{n,j}^\circ)=e_{k,j}=\binom{j}{k}$. Let $\beta^n_{k,r}:=\left(\chi^n_{k,s}-\chi^n_{k,0} \right)$ we have
\[
\left(
\begin{array}{ccccc}
\binom{1}{1} & \binom{2}{1} & \binom{3}{1}     &\cdots & \binom{n-1}{1}  \\
0 & \binom{2}{2} & \binom{3}{2} & \cdots     & \binom{n-1}{2} \\
\cdots & \cdots & \cdots & \cdots  & \cdots  \\
0 & 0  & 0 &  \cdots & \binom{n-1}{n-1}
\end{array}
\right)
\cdot 
\left(
\begin{array}{ccc}
\beta^n_{1,1} & \cdots & \beta^n_{1,n-1}   \\
\beta^n_{2,1}  &\cdots  & \beta^n_{2,n-1}  \\
\cdots & \cdots & \cdots   \\
\beta^n_{n-1,1} & \cdots & \beta^n_{n-1,n-1}
\end{array}
\right)
=
\left(
\begin{array}{ccccc}
0 & 0 &\cdots & 0        &  \binom{n-1}{n-1}  \\
0 & 0 &   \cdots & \binom{n-2}{n-2} & \binom{n-1}{n-2} \\
\cdots &   \cdots & \cdots & \cdots  & \cdots  \\
\binom{1}{1} & \binom{2}{1} & \binom{3}{1}     &\cdots & \binom{n-1}{1}  \\
\end{array}
\right) . 
\] 
Since the inverse matrix of $[\binom{i}{j}]_{1\leq i,j\leq n-1}$ is  $[(-1)^{j-i} \binom{i}{j}]_{1\leq i,j\leq n-1}$, we then have
\begin{align*}
\beta^n_{k,r}=\sum_{i=k}^{n-1} (-1)^{i-k} \binom{i}{k}\binom{r}{n-i} 
\end{align*}

Recall that Proposition~\ref{prop; csm(-1)} shows:
\[
\chi^{n}_{k, r}=   \beta^n_{k,r}+\chi^{n}_{k, 0}  
=   \beta^n_{k,r}+ \chi(\tau_{n,k}^\circ)-\sm_{n,k} \/.
\] 
From Proposition~\ref{prop; ordsm}  we have $\sm_{n,k}=(-1)^{n+1-k}\binom{n}{k}$, and hence
\begin{align*}
\chi^{n}_{k, r}
=& \sum_{i=k}^{n-1} (-1)^{i-k} \binom{i}{k}\binom{r}{n-i} +(-1)^{n -k}\binom{n}{k}+ \chi(\tau_{n,k}^\circ)\\
=& 
\begin{cases}
\sum_{i=k}^{n-1} (-1)^{i-k} \binom{i}{k}\binom{r}{n-i} +(-1)^{n-k}\binom{n}{k} & k<n-1\\
r -n +n^2-1 & k=n-1 
\end{cases}
\end{align*}
\end{proof}

\subsection{Skew-Symmetric Matrix}
Now we consider skew-symmetric case. 
When $m=2n+1$,  the total space $\PP(\wedge^2 k^{2n+1})$ is of dimension $n(2n+1)-1$. 
The rank loci $\tau^\wedge_{2n+1,2k+1}$ is dual to $\tau^\wedge_{2n+1,2n-2k-1}$, and they are of dimensions $(n-k)(2n+2k+1)-1$ and $(k+1)(4n-2k-1)-1$ respectively. When $m=2n$ is even, the total space $\PP(\wedge^2 k^{2n})$ is of dimension $n(2n-1)-1$. 
The rank loci $\tau^\wedge_{2n,2k}$ is dual to $\tau^\wedge_{2n,2n-2k}$, and they are of dimensions $(n-k)(2n+2k-1)-1$ and $(k+1)(4n-2k-1)-1$.
Thus we have the following:
\begin{theo}
\label{theo; secEulerskew}
Let $L_B$ be a hyperplane correspond to $l\in \tau_{m,B}^\circ\in \PP^{^* \binom{m}{2}-1}$. Then we have
\begin{align*}
\chi(\tau^{\wedge \circ}_{2n,2k} \cap L_{2r})  
=& 
\chi(\tau^{\wedge \circ}_{2n+1,2k+1}\cap L_{2r+1}) \\
=&
\begin{cases}
\sum_{i=k}^{n-1} (-1)^{i-k+1} \binom{i}{k}\binom{r}{n-i}  + (-1)^{n-k}\binom{n}{k}  & k<n-1\\
 2n^2-1-r & k=n-1 
\end{cases}
\end{align*}
Here we make the convention that $\binom{a}{b}=0$ whenever $b<0$.
\end{theo}
\begin{proof}
First we consider the odd case $m=2n+1$.
Theorem~\ref{theo: Radon} shows that
\[
(Eu_{\tau^\wedge_{2n+1,2k+1}})^\vee=(-1)\cdot Eu_{\tau^\wedge_{2n+1,2n-2k-1}}+e\ind_{\PP(\wedge^2 k^{2n+1})}
\]
Thus for $1\leq k,s\leq n-1$   we have
\[
\sum_{r=k}^{n-1}Eu_{\tau^\wedge_{2n+1, 2k+1}} (\tau^{\wedge \circ}_{2n+1, 2r+1}) \left(\chi^{2n+1}_{2r+1, 2s+1}-\chi^{2n+1}_{2r+1, 1}\right)
= -Eu_{\tau^\wedge_{2n+1, 2n-2k-1}} (\tau^{\wedge \circ}_{2n+1,2s+1}) \/.
\]
Define  $\beta^{2n+1}_{2k+1, 2s+1}=\chi^{2n+1}_{2r+1, 2s+1}-\chi^{2n+1}_{2r+1, 1}$. Recall from Theorem~\ref{theo; SkewEuler} that
\[
Eu_{\tau^\wedge_{2n+1, 2k+1}} (\tau^{\wedge \circ}_{2n+1, 2r+1})=\binom{r}{k} \/.
\]
Then we have
\[
\left(
\begin{array}{ccccc}
\binom{1}{1} & \binom{2}{1} & \binom{3}{1}     &\cdots & \binom{n-1}{0}  \\
0 & \binom{2}{2} & \binom{3}{2} & \cdots     & \binom{n-1}{2} \\
\cdots & \cdots & \cdots & \cdots  & \cdots  \\
0 & 0  & 0 &  \cdots & \binom{n-1}{n-1}
\end{array}
\right)
\cdot 
\left(
\begin{array}{ccc}
\beta^{2n+1}_{3,1}  &\cdots  & \beta^{2n+1}_{3,2n-1}  \\
\cdots & \cdots & \cdots   \\
\beta^{2n+1}_{2n-1,1} & \cdots & \beta^{2n+1}_{2n-1,2n-1}
\end{array}
\right)
=
-\left(
\begin{array}{ccccc}
0 & 0 &\cdots & 0        &  \binom{n-1}{n-1}  \\
0 & 0 &   \cdots & \binom{n-2}{n-2} & \binom{n-2}{n-1} \\
\cdots &   \cdots & \cdots & \cdots  & \cdots  \\
\binom{1}{1} & \binom{2}{1} & \binom{3}{1}     &\cdots & \binom{n-1}{0}  
\end{array}
\right) . 
\] 
Since the inverse matrix of $[\binom{i}{j}]_{1\leq i,j\leq n-1}$ is  $[(-1)^{j-i} \binom{i}{j}]_{1\leq i,j\leq n-1}$, we then have
\begin{align*}
\beta^{2n+1}_{2k+1,2r+1}=\sum_{i=k}^{n-1} (-1)^{i-k+1} \binom{i}{k}\binom{r}{n-i}
\end{align*}
Thus 
\begin{align*}
\chi^{2n+1}_{2k+1, 2r+1}= &  \beta^{2n+1}_{2k+1,2r+1}+\chi^{2n+1}_{2k+1, 1} \\
=& \sum_{i=k}^{n-1} (-1)^{i-k+1} \binom{i}{k}\binom{r}{n-i}+\chi^{2n+1}_{2k+1, 1}  \\
=& \sum_{i=k}^{n-1} (-1)^{i-k+1} \binom{i}{k}\binom{r}{n-i}+\chi_{\tau^{\wedge \circ}_{2n+1,2k+1}} -\sm_{2n+1, 2k+1} \\
=& 
\begin{cases}
\sum_{i=k}^{n-1} (-1)^{i-k+1} \binom{i}{k}\binom{r}{n-i}  + (-1)^{n-k}\binom{n}{k}  & k<n-1\\
 2n^2-1-r & k=n-1 
\end{cases}
\end{align*}
The last step follows from Proposition~\ref{prop; smSkew} we have the formula in the statement. 
When $m=2n$ is even, Theorem~\ref{theo: Radon} shows that
\[
(Eu_{\tau^\wedge_{2n+1,2k+1}})^\vee=(-1)\cdot Eu_{\tau^\wedge_{2n+1,2n-2k-1}}+e\ind_{\PP(\wedge^2 k^{2n+1})}
\]
The same argument shows that
\begin{align*}
\chi^{2n}_{2k, 2r}= &   
\begin{cases}
\sum_{i=k}^{n-1} (-1)^{i-k+1} \binom{i}{k}\binom{r}{n-i}+ (-1)^{n-k}\binom{n}{k} & k<n-1\\
 2n^2-1-r & k=n-1 
\end{cases}
\end{align*}
\end{proof}

\subsection{Symmetric Matrix}
Now we compute the sectional Euler characteristic of symmetric rank loci. We have the following theorem.
\begin{theo}
\label{theo; sectionalSym}
Let $\tau_{m,A}^{S \circ}$ be the  orbits  in $\PP(Sym^2 k^m)$. For any $l_B\in (\tau^{S \circ}_{m,B})\subset \PP(Sym^2k^{m \vee}) $, let $L_B$ be the corresponding hyperplane. Then for $A,B$ from $1$ to $m-1$ we have
\begin{align*}
\chi(\tau_{m,A}^{S \circ}\cap L_B)=
\begin{cases}
\sum_{l=1}^n (-1)^{l-k} \binom{l}{k}\binom{r}{n+1-l}-(-1)^{n-k}\binom{n+1}{k} +\chi^S_{2n+2,2k} &   A=2k, B=2r, m=2n+2   \\
\chi^S_{2n+2,2k+1} &  A=2k+1, B=2r, m=2n+2   \\
\sum_{l=1}^n (-1)^{l-k} \binom{l}{k}\binom{r+1}{n+1-l} - (-1)^{n-k}\binom{n+1}{k} +\chi^S_{2n+2,2k} &   A=2k, B=2r+1, m=2n+2  \\
\sum_{l=1}^n (-1)^{l-k+1} \binom{l}{k}\binom{r}{n-l}+\chi^S_{2n+2,2k} & A=2k+1, B=2r+1  , m=2n+2  \\
\sum_{l=1}^n (-1)^{l-k} \binom{l}{k}\binom{r}{n+1-l} - (-1)^{n-k}\binom{n}{k} + \chi^S_{2n+1,2k} &  A=2k   B=2r, m=2n+1   \\
 (-1)^{n-k+1}\binom{n}{k} +\chi^S_{2n+1,2k+1} &  A=2k+1, B=2r, m=2n+1   \\
 \sum_{l=0}^{n-1} (-1)^{l-k+1} \binom{l}{k-1}\binom{r}{n-l} - (-1)^{n-k}\binom{n}{k} +\chi^S_{2n+1,2k} &   A=2k, B=2r+1, m=2n+1 \\
\sum_{l=1}^{n-1} (-1)^{l-k} \binom{l}{k}\binom{r}{n-l} - (-1)^{n-k+1}\binom{n}{k} +\chi^S_{2n+1,2k+1} & A=2k+1, B=2r+1, m=2n+1  
\end{cases}
\end{align*}
Here $\chi^S_{m,A}$ are defined in Proposition~\ref{prop; rankloci}-$vii$. 
 We make the convention that $\binom{a}{b}=0$ whenever $b<0$.
\end{theo}
\begin{proof}
We follow the proof as in ordinary matrix case. Recall that for $k=1,\cdots ,m-1$ the variety $\tau^S_{m,k}$ is dual to $\tau^S_{m,m-k}$, with dimensions $\frac{(m-k)(m+k+1)}{2}-1$ and $\frac{k(2m-k+1)}{2}-1$ respectively. Here $\PP(Sym^2 k^m)=\PP^{N}$, where$N=\binom{m+1}{2}-1$. Thus from Theorem~\ref{theo: Radon} we have
\[
(Eu_{\tau^S_{m,k}})^\vee=(-1)^{(m-k)(m+1)} Eu_{\tau^S_{m,m-k}}+e\ind_{\PP^N}
\]
We denote $\beta^{m}_{i,j}$ to be the difference $\chi({\tau_{m,j}^{S \circ}\cap L_j})-\chi({\tau_{m,j}^{S \circ}\cap H})$, where $L_j\in \tau_{m,j}^{S \circ}$ in the dual space and $H$ is a generic hyperplane.
First we consider the case that $m=2n+2$ is even. 
Apply Theorem~\ref{theo; SymEuler} we have
\begin{small}
\[
\left(
\begin{array}{cccccc}
\binom{0}{0} & 0 & \binom{1}{0} & 0     &\cdots & \binom{n}{0}  \\
0 & \binom{1}{1} & \binom{1}{1} & \cdots & \binom{n}{1} & \binom{n}{1} \\
\cdots & \cdots & \cdots & \cdots & \cdots  & \cdots  \\
0 & 0  & 0 & 0 & \cdots & \binom{n}{n}
\end{array}
\right)
\cdot 
\left(
\begin{array}{ccc}
\beta^{2n+2}_{1,1} & \cdots & \beta^{2n+2}_{1,2n+1}   \\
\beta^{2n+2}_{2,1}  &\cdots  & \beta^{2n+2}_{2,2n+1}  \\
\cdots & \cdots & \cdots   \\
\beta^{2n+2}_{2n+1,1} & \cdots & \beta^{2n+2}_{2n+1,2n+1}
\end{array}
\right)
=
\left(
\begin{array}{cccccc}
0 & 0 &\cdots & 0 & 0      & -\binom{n}{n}  \\
\cdots & \cdots & \cdots & \cdots & \cdots  & \cdots  \\
0 & \binom{1}{1} & \binom{1}{1} & \cdots & \binom{n}{1} & \binom{n}{1} \\
-\binom{0}{0} & 0  & -\binom{1}{0} & 0 & \cdots & -\binom{n}{0}
\end{array}
\right) . 
\] 
\end{small}
\begin{lemm}
The inverse matrix of 
\[
\left(
\begin{array}{cccccc}
\binom{0}{0} & 0 & \binom{1}{0} & 0     &\cdots & \binom{n}{0}  \\
0 & \binom{1}{1} & \binom{1}{1} & \cdots & \binom{n}{1} & \binom{n}{1} \\
\cdots & \cdots & \cdots & \cdots & \cdots  & \cdots  \\
0 & 0  & 0 & 0 & \cdots & \binom{n}{n}
\end{array}
\right)
\] is 
\[
\left(
\begin{array}{cccccccc}
\binom{0}{0} & 0 & -\binom{1}{0} & 0  & \binom{2}{0} & 0    &\cdots & \binom{n}{0}  \\
0 & \binom{1}{1} & -\binom{1}{1} &   -\binom{2}{1}& \binom{2}{1}&\cdots & -\binom{n}{1} & \binom{n}{1} \\
0 & 0 & \binom{1}{1} & 0 & -\binom{2}{1} & 0      &\cdots & -\binom{n}{1}  \\

\cdots & \cdots & \cdots & \cdots & \cdots & \cdots & \cdots  & \cdots  \\
0 & 0 & 0 & 0  & 0 & 0 & \cdots & \binom{n}{n}
\end{array}
\right)
\]
\end{lemm} 
The proof of the Lemma is straight binomial computation, and we omit it here. 
For any $A,B$ from $1$ to $2n+1$  we then have:
\begin{align*}
\beta^{2n+2}_{A,B}=
\begin{cases}
\sum_{l=1}^n (-1)^{l-k} \binom{l}{k}\binom{r}{n+1-l}   &   A=2k, B=2r\\
 0 &  A=2k+1, B=2r \\
\sum_{l=1}^n (-1)^{l-k} \binom{l}{k}\binom{r+1}{n+1-l}   &   A=2k, B=2r+1\\
\sum_{l=1}^n (-1)^{l-k+1} \binom{l}{k}\binom{r}{n-l} & A=2k+1, B=2r+1 
\end{cases} 
\end{align*}
For the case that $m=2n+1$, notice that
\[
(Eu_{\tau^S_{2n+1,k}})^\vee= Eu_{\tau^S_{2n+1,2n+1-k}}+e\ind_{\PP^N}
\]
Thus we have
\[
\left(
\begin{array}{cccccc}
\binom{0}{0} & 0 & \binom{1}{0} & 0     &\cdots & \binom{n}{0}  \\
0 & \binom{1}{1} & \binom{1}{1} & \cdots & \binom{n}{1} & \binom{n}{1} \\
\cdots & \cdots & \cdots & \cdots & \cdots  & \cdots  \\
0 & 0  & 0 & 0 & \cdots & \binom{n}{n}
\end{array}
\right)
\cdot 
\left(
\begin{array}{ccc}
\beta^{2n+1}_{1,1} & \cdots & \beta^{2n+1}_{1,2n+1}   \\
\beta^{2n+1}_{2,1}  &\cdots  & \beta^{2n+1}_{2,2n+1}  \\
\cdots & \cdots & \cdots   \\
\beta^{2n+1}_{2n+1,1} & \cdots & \beta^{2n+1}_{2n+1,2n+1}
\end{array}
\right)
=
\left(
\begin{array}{cccccc}
0 & 0 &\cdots & 0 & 0      & \binom{n}{n}  \\
\cdots & \cdots & \cdots & \cdots & \cdots  & \cdots  \\
0 & \binom{1}{1} & \binom{1}{1} & \cdots & \binom{n}{1} & \binom{n}{1} \\
\binom{0}{0} & 0  & \binom{1}{0} & 0 & \cdots & \binom{n}{0}
\end{array}
\right) . 
\] 
Mimicking the  computation for $m=2n+2$ we get
\begin{align*}
\beta^{2n+1}_{A,B}=
\begin{cases}
\sum_{l=1}^n (-1)^{l-k} \binom{l}{k}\binom{r}{n+1-l}   &   A=2k, B=2r\\
 0 &  A=2k+1, B=2r \\
 \sum_{l=0}^{n-1} (-1)^{l-k+1} \binom{l}{k-1}\binom{r}{n-l}   &   A=2k, B=2r+1\\
\sum_{l=1}^{n-1} (-1)^{l-k} \binom{l}{k}\binom{r}{n-l} & A=2k+1, B=2r+1 
\end{cases} 
\end{align*}
Recall that $\sm_{m,A}=\chi(\tau_{m,A}^{S \circ})-\chi(\tau_{m,A}^{S \circ}\cap H)$. We then have
\begin{align*}
\chi(\tau_{m,A}^{S \circ}\cap L_B)=&\beta^m_{A,B}+\chi(\tau_{m,A}^{S \circ}\cap H) \\
=& \beta^m_{A,B}+ \chi(\tau_{m,A}^{S \circ})- \sm_{m,A}    
\end{align*}
Proposition~\ref{prop; rankloci}-$7$ shows that for symmetric rank loci we have
\begin{align*}
\chi(\tau_{m,A}^{S \circ})=
\chi^S_{m,A}
\end{align*}
Thus combine with Proposition~\ref{prop; smSym} we have the formula in the theorem.
\end{proof}

\section{Stalk Euler Characteristic of the Intersection Cohomology Complex of Rank loci}
\label{S; stalkRankloci}
In this section we compute the stalk Euler characteristic of the Intersection Cohomology complex of rank loci. In this section we assume $k=\CC$. 
\subsection{Index Theorem}
Let $\sqcup_{i\in I} S_i$ be a (finite) Whitney stratification of a complex projective variety $X$ in $\PP^N$. The theory of Chern-Schwartz-MacPherson classes of constructible functions  can be viewed as the pushdown of the theory of characteristic cycles of constructible sheaves. In short, there is a map $CC$ from the Grothendieck group of derived category of constructible sheaves to the group of (conical) Lagrangian cycles of $T^* \PP^N$ with support in $X$, together with the following diagram
\[
\begin{tikzcd}
 K_0(D^b_{c}(X)) \arrow{r}{CC} \arrow{d}{\chi_x} & L(X) \arrow{dl}{\sim} \arrow{r}{[*]} \arrow{d}{} & A_{m-1}(\PP(T^*\PP^N)) \arrow{d}{sh}  \\
F(X) \arrow{r}{c_*} & A_*(X) \arrow{r}{i_*}&  A_*(\PP^N) 
\end{tikzcd}
\]
Here 
$[*]$ denotes the fundamental classes. 
The first vertical map $\chi_x$ takes the stalk Euler characteristic of a constructible sheaf complex, and the last vertical
map  
$
sh\colon A_{m-1}(\PP(T^*M))\to A_{*}(M)
$
is given by `casting the shadow ' process discussed in \cite{Aluffi3}.  

In fact the theory of characteristic cycles doesn't require projectivization or compactness: the above projective assumption on $X$ is totally due to the Chern-Schwartz-MacPherson theory. The characteristic cycles on a  projective stratification behave the same with the induced conical stratification on its affine cone. For the rest of this section we only assume $X=\sqcup_{i\in I} S_i\subset M$  is a Whitney stratified space.
For any stratum $S_i$, there is a naturally assigned constructible sheaf called the Intersection Cohomology Sheaf complex of $S_i$. 
We denote it by  $\IC_{\overline{S}_i}$, and call it the $\IC$ sheaf of $S_i$ for short.  This  complex can be obtained by a sequence of (derived truncated) pushforward of the local system on $S_i$ to  its closure along all the smaller strata. 
For details about the intersection cohomology sheaf and intersection homology we refer to \cite{G-M1} and \cite{G-M2}. 
For any constructible sheaf $\cF^\bullet$ on $X$, its characteristic cycle  has the form 
\[
CC(\cF^\bullet)=\sum_{i\in I} r_i(\cF^\bullet)[T^*_{\overline{S}_i}M] \/.
\]  
Here for any $i$,  $T^*_{\overline{S}_i}M$ is the conormal space of $\overline{S}_i$ defied as the closure of 
\[
\{(x, \lambda)| s\in S_i, \lambda (T_x S_i)=0\}\subset T^* M \/.
\]
The integer coefficients $r_i(\cF^\bullet)$ are called the \textit{Microlocal Multiplicities}: these are very important invariants in microlocal geometry, for details we refer to \cite[Chapter 9]{KS}. 
The following deep theorem from \cite[Theorem 3]{Dubson}, \cite[Theorem 6.3.1]{Ka} reveals the relation among the microlocal multiplicities  , the stalk Euler characterictic, and the local Euler obstructions :
\begin{theo}[Microlocal Index Formula]
\label{theo; microlocal}
For any $i,j\in I$, and any $x\in S_i$ we define 
\[
\chi_i(\cF^\bullet)=\sum_{p} (-1)^p\dim H^p(\cF^\bullet(x))
\]
to be the stalk Euler characteristic of $\cF^\bullet$. Let $e(j,i)=Eu_{\overline{S_i}}(S_j)$   be the local Euler obstructions. 
We have the following formula:
\[
\chi_j(\cF^\bullet)=\sum_{i\in I} (-1)^{\dim S_i} e(j,i)r_i(\cF^\bullet)  \/.
\]
\end{theo}
This theorem suggests that if one knows about any two sets of the indexes, then one can compute the third one. 

\subsection{Main Result}
Now we consider rank loci. In \cite{Braden} the authors studied concretely the category of perverse sheaves on rank loci, and showed the following properties in \cite[Corollary 4.10, 4.11]{Braden}: 
\begin{prop}
For ordinary  matrix  rank loci with $i\geq 1$, and for skew-symmetric  matrix rank loci with $n-i< 2\lfloor \frac{n}{2} \rfloor$
, we have
\[
\pi_{1}(\Sigma^\circ_i)=\pi_{1}(\Sigma^{\wedge \circ}_i) =1
\]
Thus there are only trivial local systems. 
The characteristic cycles of the $\IC$ complexes are then irreducible, i.e., for $j\geq 1$ we have
\[
CC(\IC_{\Sigma_{n,j}}(\CC))=[T_{\Sigma_{n,j}}^* \CC^{n^2}];  \quad CC(\IC_{\Sigma^\wedge_{n,j}}(\CC))=[T_{\Sigma^\wedge_{n,j}}^* \CC^{\binom{n}{2}}] \/.
\]
For symmetric rank loci, when $n-1\geq i\geq 1$ we have
\[
\pi_{1}(\Sigma^{S \circ}_i) = \ZZ/2\ZZ;
\]
Thus for $n-1\geq i\geq 1$ there is a unique non-trivial rank $1$  local system $\cL_i$ on each $\Sigma^{S \circ}_i$, such that  $\cL_i^{\otimes 2}=\CC$. The characteristic cycles of the $\IC$ complexes are given by
\begin{align*}
CC(\IC_{\Sigma^S_{n,j}}(\CC)) =&
\begin{cases}
[T_{\Sigma^S_{n,j}}^* \CC^{\binom{n+1}{2}}] & j=2k \\
[T_{\Sigma^S_{n,j}}^* \CC^{\binom{n+1}{2}}] +  [T_{\Sigma_{n,j+1}}^* \PP^N]  & j=2k+1 \\
\end{cases}  \\
CC(\IC_{\Sigma^S_{n,j}}(\cL_j))
=&
\begin{cases}
[T_{\Sigma^S_{n,j}}^* \CC^{\binom{n+1}{2}}] & j=2k+1 \\
[T_{\Sigma^S_{n,j}}^* \CC^{\binom{n+1}{2}}] +[T_{\Sigma_{n,j+1}}^* \PP^N]  & j=2k
\end{cases} ;
\end{align*}
For $j=n$ we have $\Sigma^S_{n,n}=\{0\}$  and $CC(\IC_{\Sigma_{n,n}}) =[T_{0}^* \CC^{\binom{n+1}{2}}]$.
\end{prop}

Thus combine the index theorem with above proposition, with the knowledge in \S\ref{S; EulerRankLoci}  we obtain the following:
\begin{theo}[Stalk Euler Characteristic of $\IC$ complex]
\label{theo; stalkRankloci}
Let $\IC_{\Sigma^*_{n,k}(\CC)}$ $($respectively, $\IC_{\Sigma^S_{n,k}}(\cL_k))$  be the intersection cohomology sheaf complex of $\Sigma^*_{n,k}$ $($respectively,  $\Sigma^S_{n,k} )$. 
We use $\chi_{\Sigma^{* \circ}_{n,r}}$ to denote $\chi_p$ for any $p\in \Sigma^{* \circ}_{n,r}$.
For ordinary matrix rank loci $( *=\emptyset )$, we have
\[
(-1)^{\dim \Sigma_{n,i}}\chi_{\Sigma^\circ_{n,r}}(\IC_{\Sigma_{n,k}}(\CC))=\binom{r}{k} \/.
\]
For skew-symmetric matrix rank loci $(*=\wedge)$, we have:
\begin{align*}
(-1)^{\dim \Sigma^{\wedge }_{2n,2k}}\chi_{\Sigma^{\wedge \circ}_{2n,2r}}(\IC_{\Sigma^\wedge_{2n,2k}})=(-1)^{\dim \Sigma^{\wedge  }_{2n+1,2k+1}} \chi_{\Sigma^{\wedge \circ}_{2n+1,2r+1}}(\IC_{\Sigma^\wedge_{2n+1,2k+1}})=\binom{r}{k} 
\end{align*}
For symmetric matrix rank loci $(*=S)$, we have:
\begin{align*}
(-1)^{\dim \Sigma^{S }_{n,i}} \chi_{\Sigma^{S \circ}_{n,j}}(\IC_{\Sigma^S_{n,i}}(\cL_i))
=& 
\begin{cases}
 \binom{k}{s} & (i,j)=(2s,2k) \\
 0  & (i,j)=(2s+1,2k)   \\
 0  & (i,j)=(2s,2k+1) \\
 \binom{k}{s} & (i,j)=(2s+1,2k+1)  
\end{cases}
\\
(-1)^{\dim \Sigma^{S }_{n,i}} \chi_{\Sigma^{S \circ}_{n,j}}(\IC_{\Sigma^S_{n,i}}(\CC))
=& 
\begin{cases}
 \binom{k}{s} & (i,j)=(2s,2k) \\
 \binom{k}{s+1} & (i,j)=(2s+1,2k) \\
 \binom{k+1}{s+1} & (i,j)=(2s+1,2k) \\
 \binom{k}{s} & (i,j)=(2s+1,2k+1)  
\end{cases}   
\end{align*}
\end{theo}
\begin{proof}
Since the proof for two formulas are almost identical, we only prove the one for $\chi_{\Sigma^{S \circ}_{n,j}}(\IC_{\Sigma^S_{n,i}}(\cL_i))$.  Let $d_{n,j}=\dim \Sigma^{S \circ}_{n,j}$, we have 
\begin{align*}
\chi_{\Sigma^{S \circ}_{n,r}}(\IC_{\Sigma^S_{n,k}}(\cL_i))
=& \sum_{j} (-1)^{d_{n,j}} Eu_{\Sigma^{S }_{n,j}}(\Sigma^{S \circ}_{n,r}) \cdot r_{j}(\IC_{\Sigma^S_{n,k}}(\cL_i)) \\
=& 
\begin{cases}
(-1)^{d_{n,2A}} Eu_{\Sigma^{S }_{n,2A}}(\Sigma^{S \circ}_{n,r}) +(-1)^{d_{n,2A+1}} Eu_{\Sigma^{S }_{n,2A+1}}(\Sigma^{S \circ}_{n,r}) & k=2A ;\\
(-1)^{d_{n,2A+1}} Eu_{\Sigma^{S }_{n,2A+1}}(\Sigma^{S \circ}_{n,r}) & k=2A+1 \\
\end{cases} 
\end{align*}
Recall that $d_{n,k}=k(n-k)+\binom{n-k+1}{2}$. Thus we have
\[
(-1)^{d_{n,2A+1}-d_{n,2A}}=(-1)^{k-1} \/.
\]
This proves that 
\begin{align*}
(-1)^{d_{n,k}} \chi_{\Sigma^{S \circ}_{n,r}}(\IC_{\Sigma^S_{n,k}}(\cL_i))
=& 
\begin{cases}
 Eu_{\Sigma^{S }_{n,2A}}(\Sigma^{S \circ}_{n,r}) -Eu_{\Sigma^{S }_{n,2A+1}}(\Sigma^{S \circ}_{n,r})
 & k=2A ;\\
Eu_{\Sigma^{S }_{n,2A+1}}(\Sigma^{S \circ}_{n,r}) & k=2A+1 \\
\end{cases}
\end{align*}
The formula then follows directly from Theorem~\ref{theo; SymEuler}. 
\end{proof}

\begin{rema}
As shown in \cite[\S 6.2]{G-M2}, when a complex variety $X$ admits a small resolution, the $\IC$ complex is the $($shifted$)$ derived pushforward of  the constant sheaf from the small resolution.
Thus for any point $x\in X$, the signed stalk Euler characteristic of $\IC_X$ equals the Euler characteristic of the fiber:
$$
(-1)^{\dim X_i} \chi_x (\IC_X)=\chi (p^{-1}(x)).
$$
For ordinary matrix rank loci $\tau_{n,k}$, the Tjurina transform $\hat{\tau}_{n,k}$ is exactly a small resolution. The fiber at $\tau_{n,r}$ is isomorphic to $G(k,r)$, which has Euler characteristic $\binom{r}{k}$. For symmetric and skew-symmetric rank loci, the Tjurina transforms fail to be small: the fiber Euler characteristics are too large. 
\end{rema}

In fact, assuming that there are canonically define small resolutions for symmetric rank loci, then the value
$\chi_{\Sigma^{S \circ}_{n,2k+1	}}(\IC_{\Sigma^S_{n,2s+1}}(\CC))$ forces the fiber Euler characteristics to jump. This supports the following conjecture.
\begin{conj}
There do not exist canonical small resolutions for symmetric rank loci.
\end{conj}

\section{An Enumerative Problem}
\label{S; enumaration}
In this section we finish the proof of Proposition~\ref{prop; smSym}. First
we consider the following situation. Let $V$ be a complex vector space of dimension $n+1$, and let $X\subset \PP(V)$ be an irreducible quadratic hypersurface defined by the vanishing of $f_X$.
Let  $\mathbb{G}(d,n)=G(d+1,n+1)$ be the Grassmannian of projective $d$ planes, with $S$ and $Q$ being universal sub and quotient bundles.
 We consider the space of projective $d$ planes contained in $X$:
\[
Z^X_{d,n}:=\{\Lambda| \Lambda\subset X\}\subset \mathbb{G}(d,n)
\] 
This space is called the Fano scheme of $X$.
Notice that quadratic hypersurfaces in $\PP(V)$ correspond to sections of the sheaf $Sym^2(V^\vee)$, thus the function defining $X$ induces a section $s_X$ from $\mathbb{G}(d,n)$ to $Hom(Sym^2(S), \CC)$, sending $\Lambda$ to 
\[
{s_X}|_{\Lambda}\colon Sym^2 \Lambda \subset  Sym^2 V\to \CC \/.
\]
The last map is defined by $f_X$. One can check that the zero locus of $s_X$ is exactly supported on $Z^X_{d,n}$. 
When $2d\leq n-1$, if we choose a generic smooth section $f_X$ in $Sym^2(V^\vee)$,  the space $Z^X_{d,n}$ is smooth of codimension $\binom{d+2}{2}$. We denote this space  by $Z_{d,n}$, since for generic sections $f_X$ its topological structure is preserved. For details we refer to \cite{Altman-Kleiman}. 

In classical enumerative geometry we concern the degree of this variety, but here we consider its (topological) Euler characteristic instead. When $Z_{d,n}$ is of dimension $0$, they coincide. We denote $F(d,n)=\chi(Z_{d,n})$ to be the Euler characteristic. 
\begin{theo}
When $2d\leq n-1$, 
the functions $F(d,n)$ satisfy the following recursive formula. 
\begin{align*}
F(d,2m)=& F(d,2m-1)\/, \\
F(d,2m+1)=& F(d,2m)+2\cdot F(d-1,2m) \/.
\end{align*}
Thus computing the initial value we have
\[
F(d,2m-1)=F(d,2m)=2^{d+1}\cdot\binom{m}{m-d-1} \/.
\]
\end{theo} 
\begin{proof}
First we consider the initial value $d=0$. When $d=0$ we can see that $Z_{0,n}=X$ and $F(0,n)=\chi(X)$. Notice that $X$ is a smooth quadratic hypersurface in $\PP^n$, thus we have
\[
c_{sm}^X(H)=\frac{2H\cdot (1+H)^{n+1}}{(1+2H)}=\sum_{i=0}^{n-1}(-1)\binom{n+1}{i}(-2)^{n-i} H^n+\text{ lower degree terms} \/.
\]
In particular we have
\[
\chi(X)=\sum_{i=0}^{n-1}(-1)\binom{n+1}{i}(-2)^{n-i}=\frac{(-1)^{n+1}-1-2\binom{n+1}{n}}{2}=n+1+\frac{(-1)^{n+1}-1}{2} \/.
\]
This shows that $F(0,2m)=F(0,2m-1)=2m$.

Now we prove the induction process. Choose a generic hypersurface $X$ such that $Z_{d,n}=Z^X_{d,n}$.
Notice that the intersection of $X$ with a generic hyperplane $H$ is again a generic quadratic hypersurface $Y=X\cap H$ in $H\cong \PP^{n-1}$. We define the following two strata in $Z^X_{d,n}$: 
\[
A_H:=\{\Lambda|\Lambda\subset H\} \subset Z^X_{d,n};\quad B_H:=Z^X_{d,n}\setminus A_H \/.
\]
There are two canonical restriction maps
\[
p\colon A_H\to Z^{Y}_{d,n-1} ;\quad q\colon B_H\to Z^Y_{d-1,n-1} 
\]
defined as follows. 
For $p$, we simply send $\Lambda\subset X$ to $\Lambda\subset Y=X\cap H$, since $\Lambda\subset H$. By the definition of $A_H$ one can observe that $p$ is an isomorphism, thus we have $\chi(A_H)=\chi(Z^Y_{d,n-1})=F(d,n-1)$.

We define $q$ by sending a $d$-plane $\Lambda$ to the $d-1$-plane $\Lambda\cap H$. This is well-defined since $\Lambda\in B_H$ iff $\Lambda\not\subset H$. The map $q$ is dominant for generic $X$ and $H$, and for any $\Gamma\in Z^Y_{d-1,n-1}$ the fiber $q^{-1}(\Gamma)$ is
\[
\{\Lambda\in G(d+1,n+1)|\Gamma\subset \PP(\Lambda)\subset X\}\setminus \{\Lambda\in G(d+1,n)| \Gamma\subset\PP(\Lambda)\subset X\cap H=Y\} \/.
\]
We denote the two spaces by $C_X$ and $C_Y$ respectively.
Notice that $\{\Lambda\in G(d+1,n+1)|\Gamma\subset \PP(\Lambda)\}$ can be identified with  $G(1, n-d)$. 
For any $L\in G(1, n-d)$, to force a the (projectivization of) the span of $L$ with $\Gamma$ to live in $X$, we need $d-1$ linear equations and a degree $2$ equation. Thus we can identify $C_X$ with a 
generic quadratic hypersurface in $\PP^{n-2d}$, and similarly $C_Y$ is identified with a generic quadratic hypersurface in $\PP^{n-2d-1}$. This shows that the fibers $q^{-1}(\Gamma)$ are all isomorphic, and we then have 
\[
\chi(B_H)=\chi(q^{-1}(\Gamma))\cdot F(d-1,n-1) \/.
\] 
Moreover, the generic assumption  shows that $C_X$ and $C_Y$ are both smooth. Thus similar argument of the initial case shows
\begin{align*}
\chi(q^{-1}(\Gamma))=&\chi(C_X)-\chi (C_Y) \\
=&\left( n-2d+1+\frac{(-1)^{n-2d+1}-1}{2}\right) - \left( n-2d+\frac{(-1)^{n-2d}-1}{2}  \right) \\
=& 1+ (-1)^{n-1}
\end{align*}
The  proof is then closed by 
\[
F(d,n)=\chi(Z^X_{d,n})=\chi(A_H)+\chi(B_H)=F(d,n-1)+(1+(-1)^{n-d})\cdot F(d-1,n-1) \/.
\]
\end{proof}

When $d=2r$ and $n=3r+1$, the space $Z_{k,m}$ is of dimension $0$. Thus we have a direct Corollary
\begin{coro}
For a generic quadratic hypersurface $X$ in $\PP^{3r+1}$, there are exactly $2^{d+1}\binom{A}{A-d-1}$ projective $2r$-planes living in $X$. Here $A:=2\cdot \lfloor \frac{n+1}{2} \rfloor$.
\end{coro}

Now we can finish the proof of Proposition~\ref{prop; smSym}.
\begin{coro}
\label{coro; Schubert}
\begin{align*}
\int_{G(d+1,2m+1)}\frac{c(S^\vee\otimes Q)c_{\binom{d+2}{2}}(Sym^2(S^\vee))}{c(Sym^2(S^\vee))} =& 2^{d+1}\cdot \binom{m}{m-d-1} \\
\int_{G(d+1,2m)}\frac{c(S^\vee\otimes Q)c_{\binom{d+2}{2}}(Sym^2(S^\vee))}{c(Sym^2(S^\vee))} =& 2^{d+1}\cdot \binom{m}{m-d-1}  \/.
\end{align*}
\end{coro}
\begin{proof}
When $2d\leq n-1$, recall that the scheme $Z^X_{d,n}$ is the zero locus of a section $s_X\colon G(d+1,n+1)\to Sym^2(S^\vee)$. For generic  $X$ it is smooth of codimension $\binom{k+2}{2}=\rk Sym^2(S^\vee)$. Thus we have
\[
c_{sm}^{Z^X_{d,n}}=\frac{c(S^\vee\otimes Q)c_{\binom{k+2}{2}}(Sym^2(S^\vee))}{c(Sym^2(S^\vee))} \in A_*(G(d+1,n+1)) \/.
\]
And the formula follows directly from the theorem. 

It remains to prove the following: for $2d\geq n$ we have
\[
\int_{G(d+1,n+1)}\frac{c(S^\vee\otimes Q)c_{\binom{d+2}{2}}(Sym^2(S^\vee))}{c(Sym^2(S^\vee))} = 0 \/.
\]
This follows from by standard Schubert calculus. 
First, from \cite{Lascoux} we have 
\[
c_{\binom{d+2}{2}}(Sym^2(S^\vee))=2^{d+1} \Delta_{[d+1]}(c(S^\vee)) \/;
\]
where $[d+1]$ denotes the partition $(d+1,d,d-1,\cdots ,1,0,\cdots ,0)$. 
Recall \cite[Lemma 14.5.1]{INT}, which says that if there exists some $0\leq r\leq k-1$ such that $s_i(E)=0$ for $i>r$ and $\lambda_{r+1}>0$, then
\[
\Delta_{(\lambda_1,\cdots ,\lambda_k)}(c(E))=0
\]
Notice that when $2d\geq n$, $\rk Q=n-d\leq d< d+1=\rk S$. Thus $s_i(S^\vee)=c_i(Q^\vee)=0$ for $i> n-d$. Take  $\lambda=[d+1]$, one sees that $\lambda_{n-d+1}>0$, and then 
\[
\Delta_{[d+1]}(c(S^\vee))=0\/.
\]
This completes the proof.
\end{proof}
\begin{rema}
Although we work with $\CC$ in this section, notice that Corollary~\ref{coro; Schubert} is purely about binomial identities, thus it naturally holds for any characteristic $0$ field.
\end{rema}

\begin{rema}
Despite the similarity in forms, the author doesn't know a good direct proof  to the Schubert identity appeared in the computation in the skew-symmetric case:
\begin{align*}
\int_{G(2k,2n)} \frac{c(S^\vee\otimes Q)c_{\binom{2n-2k}{2}}(\wedge^2 Q)}{c(\wedge^2 Q)} \cap [G(2k,2n)] 
=&  \sum_{r=k}^{n} (-1)^{n-k} \binom{2r}{2k}\binom{n}{r} 
\\
\int_{G(2k+1,2n+1)} \frac{c(S^\vee\otimes Q)c_{\binom{2n-2k}{2}}(\wedge^2 Q)}{c(\wedge^2 Q)} \cap [G(2k+1,2n+1)] 
=& \sum_{r=k}^{n} (-1)^{n-k} \binom{2r+1}{2k+1}\binom{n}{r} 
\end{align*}
It should be interesting to find out what enumerative problems correspond to them.
\end{rema}
\begin{rema}
In fact, the author does not know a direct Schubert calculus proof of the Schubert integration equalities for the skew-symmetric and the symmetric case. Based on the computation for small Grassmannian we believe that the clean forms come from the symmetry and certain vanishing property from Schubert calculus. We will discuss more about the formulas in another paper. 
\end{rema}

\bibliographystyle{plain}
\bibliography{ref}

\end{document}